%% file: UnivAcylActions7-2020.tex
\documentclass[11pt,oneside]{amsart}

\usepackage{amsmath,ifthen, amsfonts, amssymb,
srcltx, amsopn, color, enumerate} 
\usepackage[linktocpage=true]{hyperref}
\usepackage[cmtip,arrow]{xy}
\usepackage{pb-diagram, pb-xy}
\usepackage{overpic}
\usepackage{microtype}

\usepackage[T1]{fontenc}

\newcommand{\showcomments}{no}

\newsavebox{\commentbox}

\newtheorem{thm}{Theorem}[section]
\newtheorem{lem}[thm]{Lemma}

\newtheorem{cor}[thm]{Corollary}

\newtheorem{prop}[thm]{Proposition}

\newtheorem{thmi}{Theorem}

\newtheorem{questioni}[thmi]{Question}

\theoremstyle{definition}
\newtheorem{defn}[thm]{Definition}
\newtheorem{rem}[thm]{Remark}

\newtheorem{exmp}[thm]{Example}

\newtheorem{notation}[thm]{Notation}

\newtheorem{claim*}{Claim}

\DeclareMathOperator{\Aut}{Aut}

\DeclareMathOperator{\diam}{\textup{\textsf{diam}}}

\newcommand{\neb}{\mathcal N}

\makeatletter

\newcommand{\Rmnum}[1]{\mathbf{{\expandafter\@slowromancap\romannumeral #1@}}}

\makeatother

\newcommand{\tup}[1]{\vec{#1}}

\let\oldmarginpar\marginpar
\renewcommand\marginpar[1]{\-\oldmarginpar[\raggedleft\footnotesize #1]{\raggedright\footnotesize #1}}

\newcommand{\tsh}[1]{\left\{\kern-.7ex\left\{#1\right\}\kern-.7ex\right\}}
\newcommand{\Tsh}[2]{\tsh{#2}_{#1}}
\newcommand{\ignore}[2]{\Tsh{#2}{#1}}

\newcommand{\co}{\colon}

\newcounter{enumitemp}

\newcommand{\dist}{\textup{\textsf{d}}}

\newcommand{\cuco}[1]{{\mathcal #1}}

\newcommand{\fontact}{{\mathcal C}}

 \usepackage{mathabx}
\newcommand{\propnest}{\sqsubsetneq}

\newcommand{\nest}{\sqsubseteq}
\newcommand{\orth}{\bot}
\newcommand{\transverse}{\pitchfork}

\newcommand{\induced}{^{*}}
\newcommand{\inducedS}{^{\tiny{\diamondsuit}}}

\newcommand{\X}{\mathcal X}
\newcommand{\s}{\mathfrak S}
\newcommand{\W}{\mathfrak W}

\newcommand{\frakT}{\mathfrak T}
\newcommand{\frakS}{\mathfrak S}

\newcommand{\FU}[1]{ \mathbf F_{#1}}
\newcommand{\wideFU}[1]{ \widehat{\mathbf {F}}_{#1}}
\newcommand{\EU}[1]{ \mathbf E_{#1}}
\newcommand{\PU}[1]{ \mathbf P_{#1}}

\newcommand{\trans}{\pitchfork}

\begin{document}
\title[Acylindrical actions and Stability in HHG]{ 
Largest acylindrical actions and Stability\\ in hierarchically hyperbolic groups}

\author[C. Abbott]{Carolyn Abbott}
\address{Columbia University, New York, New York, USA}
\email{abbott@math.columbia.edu}
\author[J. Behrstock]{Jason Behrstock}
\address{Lehman College and The Graduate Center, CUNY, New York, New York, USA}
\email{jason.behrstock@lehman.cuny.edu}
\author[M. Durham]{Matthew Gentry Durham}
\address{University of California, Riverside, Riverside, California, USA}
\email{matthew.durham@ucr.edu}

\maketitle

\centerline{
\textit{\footnotesize{
With an appendix by DANIEL BERLYNE and JACOB RUSSELL}
}}

\date{\today}

\begin{abstract}
    We consider two manifestations of non-positive
    curvature: acylindrical actions (on hyperbolic spaces) and quasigeodesic stability.  We
    study these properties for the class of hierarchically hyperbolic
    groups, which is a general framework for simultaneously studying
    many important families of groups, including mapping class groups,
    right-angled Coxeter groups, most $3$--manifold groups,
    right-angled Artin groups, and many others.
    
    A group that admits an acylindrical action on a hyperbolic space
    may admit many such actions on different hyperbolic spaces.  It is
    natural to try to develop an understanding of all such actions and
    to search for a ``best'' one.  The set of all cobounded
    acylindrical actions on hyperbolic spaces admits a natural poset
    structure, and in this paper we prove that all hierarchically
    hyperbolic groups admit a unique action which is the largest in
    this poset.  The action we construct is also universal in the
    sense that every element which acts loxodromically in some
    acylindrical action on a hyperbolic space does so in this one.
    Special cases of this result are themselves new and interesting.
    For instance, this is the first proof that right-angled Coxeter
    groups admit universal acylindrical actions.
    
    The notion of quasigeodesic stability of subgroups provides a
    natural analogue of quasiconvexity which can be considered outside
    the context of hyperbolic groups.  In this paper, we provide a
    complete classification of stable subgroups of
    hierarchically hyperbolic groups, generalizing and extending
    results that are known in the context of mapping class groups
    and right-angled Artin groups.  Along the way, we provide a
    characterization of contracting quasigeodesics; interestingly, in
    this generality the proof is much simpler than in the 
    special cases where it was already known.
	
	In the appendix, it is verified that any space satisfying the 
	{\it a priori} weaker 
	property of being an ``almost hierarchically hyperbolic space'' 
	is actually a hierarchically hyperbolic space. The results of the 
	appendix are used to streamline the proofs in the main text.
\end{abstract}

\tableofcontents

\section{Introduction}
Hierarchically hyperbolic groups were recently introduced 
by Behrstock, Hagen, and Sisto \cite{BehrstockHagenSisto:HHS_I} to provide a uniform framework in which to  
study many important families of groups, 
including mapping class groups of finite type surfaces, right-angled Coxeter groups, most 
$3$--manifold groups, right-angled Artin groups and many others. 
A \emph{hierarchically hyperbolic space} consists of: a quasigeodesic space, 
$\cuco X$; a set of \emph{domains}, $\frak S$, which index a collection 
of $\delta$--hyperbolic spaces to which $\cuco X$ projects; and, 
some additional information about these projections, including, for instance, a 
partial order on the domains and a unique largest element in that 
order, 
which we denote by $S$ (i.e., $S$ is comparable to and larger than every other domain in $\frak S$).

\subsection*{Largest acylindrical actions}

The study of acylindrical actions on hyperbolic spaces, as 
initiated in its current form by Osin \cite{Osin:acyl} building on earlier work of Sela \cite{Sela:acylindrical} and Bowditch \cite{Bowditch:tight},  
has proven to be a powerful tool for studying groups with some 
aspects of non-positive curvature. As established in 
\cite{BehrstockHagenSisto:HHS_I}, non-virtually cyclic hierarchically hyperbolic groups  
admit non-elementary acylindrical actions when 
the $\delta$--hyperbolic space associated to the maximal element in 
$\frak S$ has infinite diameter, a property which holds in all the 
above examples except for those that are direct products.

Any given group with an acylindrical action may actually admit 
many acylindrical actions on many different spaces. A natural 
question is to try and find a ``best'' acylindrical action. 
There are different ways that one might try to optimize the 
acylindrical action. For instance, the notion of a \emph{universal 
acylindrical action}, for a given group $G$, is an acylindrical action  
on a hyperbolic space $X$ 
such that every element of $G$ which acts loxodromically in some acylindrical 
action on some hyperbolic space, must act loxodromically in its action on $X$. As established by Abbott, there exist finitely generated groups which 
admit acylindrical actions, but no universal acylindrical action 
\cite{Abbott:notuniversal}; we also note that universal actions need 
not be unique \cite{AbbottBalasubramanyaOsin:extensions}. 

In \cite{AbbottBalasubramanyaOsin:extensions}, Abbott, Balasubramanya, and Osin introduce a partial order
on cobounded acylindrical actions which, in a certain sense, encodes how much information the action provides about the group.  When there exists an element in this
partial ordering which is comparable to and larger than all other
elements it is called a \emph{largest} action.  By construction, any
largest action is necessarily a universal acylindrical action and unique.

In this paper we construct a largest action for every hierarchically hyperbolic group.  Special cases of this theorem recover some recent 
results of \cite{AbbottBalasubramanyaOsin:extensions}, as well as a 
number of new cases. For instance, 
in the case of  
right-angled Coxeter groups (and more generally for special cubulated groups), 
even the existence of a 
universal acylindrical action was unknown. 
Further, outside of the relatively 
hyperbolic setting, our result provides a single 
construction that simultaneously covers these new cases as well as  
all previously known largest and 
universal acylindrical actions of finitely presented groups. The following summarizes the main 
results of Section~\ref{sec:uaa}  (where there are also 
further details on the background and comparison with known results).

\begin{thmi} [HHG have actions that are largest and universal]\label{thmi:largestaction}
    Every hierarchically hyperbolic group 
    admits a largest acylindrical action.
    In particular, the following admit acylindrical actions 
    which are largest and universal:
    \begin{enumerate}[(1)]
    \item Hyperbolic groups.
    \item Mapping class groups of orientable surfaces of finite type. 

    \item Fundamental groups of compact three-manifolds with no Nil or Sol in
    their prime decomposition.
    \item Groups that act properly and
    cocompactly on a special CAT(0) cube complex, and more generally any 
cubical group which admits a factor system.  This includes right
angled-Artin groups, right-angled Coxeter groups, and many other 
examples as in \cite{HagenSusse:HHScubical}.
    \end{enumerate} 
\end{thmi}

We use this construction of a largest action to characterize stable subgroups (Theorem \ref{thmi:stab})  and contracting elements (Corollary \ref{cor:contracting}) of hierarchically hyperbolic groups, and to describe random subgroups of hierarchically hyperbolic groups (Theorem \ref{thmi:random}).

\subsection*{Stability in hierarchically hyperbolic groups}

One of the key features of a Gromov hyperbolic space is that every
geodesic is uniformly \emph{Morse}, a property also known as 
\emph{(quasigeodesic)  
stability}; that is, any uniform
quasigeodesic beginning and ending on a geodesic must lie uniformly
close to it.  In fact, any geodesic metric space in which each
geodesic is uniformly Morse is hyperbolic.

In the context of geodesic metric spaces, the presence of Morse 
geodesics has important structural consequences for the space; for
instance, any asymptotic cone of such a space has global cut points
\cite{DrutuMozesSapir:Divergence}.  Moreover, quasigeodesic stability in groups is
quite prevalent, since any finitely generated acylindrically
hyperbolic group contains Morse geodesics \cite{Osin:acyl,
Sisto:qconvex}.

There has been much interest in developing alternative
characterizations \cite{DrutuMozesSapir:Divergence, charneysultan,
ACGH, aougab2016middle} and understanding this phenomenon in various
important contexts \cite{Minsky:quasiproj, Behrstock:asymptotic,
DrutuMozesSapir:Divergence, DurhamTaylor:stable, aougab2016middle}.
This includes the theory of Morse boundaries, which encode all
Morse geodesics of a group \cite{charneysultan, cordes2015morse,
cordes2016stability, cordes2016boundary, cashen2017metrizable}.
In \cite{DurhamTaylor:stable}, Durham and Taylor generalized the notion of stability to subspaces and subgroups.

In this paper, we obtain a complete characterization of stability in hierarchically
hyperbolic groups.

Let $(\cuco X, \mathfrak S)$ be an HHS.  We say that a subset $\cuco Y\subset \cuco X$ has 
\emph{$D$--bounded projections} when $\diam_{\fontact U}(\pi_{U}(\cuco Y))<D$ for
all non-maximal $U\in\frak S$; when the constant does not matter, we 
simply say the subset has \emph{uniformly bounded projections}.

\begin{thmi}[Equivalent conditions for subgroup stability]\label{thmi:stab}
    Any hierarchically hyperbolic group $G$ admits a hierarchically hyperbolic group structure 
    $(G,\frak S)$ such that for any finitely generated  $H <G$, the following are equivalent:

    \begin{enumerate}
    \item\label{thmi:stab:enum1} $H$ is stable in $G$;
    \item\label{thmi:stab:enum2} $H$ is undistorted in $G$ and has uniformly bounded projections;
    \item\label{thmi:stab:enum3} Any orbit map $H \rightarrow 
    \fontact S$ is a quasi-isometric embedding, where $S$ is the 
    $\nest$--maximal element in $\frak S$.
    \end{enumerate}
\end{thmi}

Theorem~\ref{thmi:stab} generalizes some previously known results.
In the case of mapping class groups: \cite{DurhamTaylor:stable} proved 
equivalence of (\ref{thmi:stab:enum1}) and (\ref{thmi:stab:enum3});
equivalence of (\ref{thmi:stab:enum2}) and (\ref{thmi:stab:enum3}) 
follows from the distance formula; moreover, \cite{KentLeininger:shadows,
Hamenstadt:ccc} yield that these conditions are also equivalent to
convex cocompactness in the sense of \cite{FarbMosher}.  The case of
right-angled Artin groups was studied in \cite{koberda2014geometry},
where they prove equivalence of (\ref{thmi:stab:enum1}) and
(\ref{thmi:stab:enum3}).

Section \ref{sec:stability} contains a more general version of 
Theorem~\ref{thmi:stab}, as well
as further applications, including Theorem \ref{thm:stab asdim}, 
which concerns the Morse
boundary of hierarchically hyperbolic groups and proves that all hierarchically hyperbolic groups have finite stable asymptotic dimension. 

\subsection*{On purely loxodromic subgroups}

In the mapping class group setting \cite{BBKL} proved that the 
conditions in Theorem \ref{thmi:stab} are also equivalent to being undistorted and purely 
pseudo-Anosov. Similarly, in the right-angled Artin group setting,  
it was proven in \cite{koberda2014geometry} that  (\ref{thmi:stab:enum1}) and
(\ref{thmi:stab:enum3}) are each equivalent to being purely loxodromic.

Subgroups of right-angled Coxeter groups all of whose elements act
loxodromically on the contact graph were studied in the recent
preprint \cite[Theorem~1.4]{Tran:loxcox}, who proved that property is 
equivalent to (\ref{thmi:stab:enum3}).  Since there often exist
elements in a right-angled Coxeter group which do not act
loxodromically on the contact graph, his condition is not equivalent 
to (\ref{thmi:stab:enum1}); it is the ability to change the
hierarchically hyperbolic structure as we do in
Theorem~\ref{thm:betterHHSstructure}, discussed below, which allows us to prove our more
general result which characterizes \emph{all} stable subgroups, not 
just the ones acting loxodromically on the contact graph.

Mapping class groups and right-angled Artin groups have the property 
that in their standard hierarchically hyperbolic structure they admit a 
universal acylindrical action on $\fontact S$, where $\fontact S$ is the hyperbolic space associated to the  
$\nest$--maximal domain $S$. On the other hand, right-angled Coxeter 
groups often don't admit universal acylindrical actions on $\fontact 
S$ in their standard structure. Accordingly, we believe the following 
questions are interesting. The first item would generalize the 
situation in the mapping class group as established in \cite{BBKL}, and the second item for right-angled Artin 
groups would generalize results proven in \cite{koberda2014geometry}, and  
for right-angled Coxeter groups would generalize results in \cite{Tran:loxcox}.  If the second item is true for the mapping class group, this would resolve a 
question of Farb--Mosher \cite{FarbMosher}. See also \cite[Question 1]{aougab2016middle}.

    \begin{questioni}\label{conji:equivalence} Let $(G,\frak S)$ be a
    hierarchically hyperbolic group which admits a universal
    acylindrical action on $\fontact S$, where $S$ is the $\nest$--maximal element in $\frak S$. 
    Let $H$ be a finitely generated subgroup of $G$. 
    \begin{itemize}
        \item Are the conditions in Theorem~\ref{thmi:stab} also 
    equivalent to:  $H$ is undistorted and acts purely
    loxodromically on $\fontact S$?
    
     \item Under what hypotheses on $(G,\frak S)$, are the conditions
     in Theorem~\ref{thmi:stab} also equivalent to: $H$ acts
     purely loxodromically on $\fontact S$?
    \end{itemize}
\end{questioni}

Note that in the context of Question~\ref{conji:equivalence}, an
element acts loxodromically on $\fontact S$ if and only if it has
positive translation length.  This holds since the action is
acylindrical and thus each element either acts elliptically or
loxodromically.

In an early version of this paper, we asked if the second part of
Question~\ref{conji:equivalence} held for all hierarchically
hyperbolic groups.  In the general hierarchically hyperbolic setting,
however, the undistorted hypothesis is necessary, as pointed out to us
by Anthony Genevois with the following example.  The necessity is
shown by Brady's example of a torsion-free hyperbolic group with a
finitely presented subgroup which is not hyperbolic
\cite{Brady:branchedcoverings}.  This subgroup is torsion-free and
thus purely loxodromic.  But, a subgroup of a hyperbolic group
is stable if and only if it is quasiconvex. Thus, since this subgroup is
not quasiconvex, we see that 
being purely loxodromic is strictly weaker than the
conditions of Theorem~\ref{thmi:stab}.

\subsection*{New hierarchically hyperbolic structures}

In order to establish the above results, we provide some new 
structural theorems about hierarchically hyperbolic spaces.

One of the key technical innovations in this paper is provided in 
Section~\ref{sec:structures}. There we prove 
Theorem~\ref{thm:betterHHSstructure} which allows us to modify  
a given hierarchically hyperbolic structure $(\cuco X,\frakS)$ by 
removing $\fontact U$ for some $U\in\frakS$ and, in their 
place, enlarging the space $\fontact S$.  
For instance, this is how 
we construct the space on which a hierarchically hyperbolic group has 
its largest acylindrical action.

\medskip

Another important tool is Theorem~\ref{thm:char of contracting} which 
provides a simple characterization 
of contracting geodesics in a hierarchically hyperbolic space.
  
The following is a restatement of that result in the case of groups: 

\begin{thmi} [Characterization of 
    contracting quasigeodesics] \label{thmi:char of contracting}
    Let $G$ be a hierarchically hyperbolic group.  For
    any $D>0$ and $K\geq 1$ there exists a $D'>0$ depending only on $D$ and $G$ 
    such that the following holds for every 
    $(K,K)$--quasigeodesic $\gamma \subset \cuco X$:   
    the quasigeodesic $\gamma$ 
    has $D$--bounded projections if and only if $\gamma$ is
    $D'$--contracting.
\end{thmi}

Since the presence of a contracting geodesic implies the 
group has at least quadratic divergence, 
an immediate consequence of Theorem~\ref{thmi:char of contracting} is 
that any hierarchically hyperbolic group has quadratic divergence 
whenever $\cuco X$ projects to an infinite diameter subset of 
$\fontact S$.

As a sample application of Theorem \ref{thmi:char of contracting}
and using work of Taylor--Tiozzo \cite{TaylorTiozzo:randomqi}, we
prove the following in Section~\ref{subsec:random} as Theorem \ref{thm:random stable}.

\begin{thmi}[Random subgroups are stable]\label{thmi:random}
    Let $(\cuco X,\frakS)$ be an HHS for which   
    $\fontact S$ has infinite diameter, where $S$ is the  
    $\nest$--maximal element, 
    and consider 
    $G<\Aut(\cuco X, \mathfrak S)$ which acts properly and cocompactly 
    on~$\cuco X$. Then any $k$--generated random subgroup of $G$ 
    stably embeds in $\cuco X$ via the orbit map.
\end{thmi}

We note that one immediate consequence of this result 
is a new proof 
of a theorem of 
Maher--Sisto: 
any random subgroup of a hierarchically 
hyperbolic group which is not the direct product of two infinite 
groups is stable \cite{maher2017random}. The 
 mapping class group and right-angled Artin groups cases of this 
 result were first
 established in \cite{TaylorTiozzo:randomqi}.

Finally, at the end of the paper we discuss a technical condition on
hierarchically hyperbolic structures, called having \emph{clean
containers}.  While in Proposition~\ref{prop:cleancontainers} this
hypothesis is shown to hold for many groups, it does not hold in all
cases.  This condition was used in earlier versions of this paper in
which it was assumed for the proof of
Theorem~\ref{thm:betterHHSstructure}, and then the general result was
bootstrapped from there.  In light of
Theorem~\ref{thm:almost_HHS_are_HHS} in the Appendix, this property is
no longer required for this paper.  We keep the contents
of this section in the paper nonetheless, since they have found 
independent interest and already been used elsewhere, e.g.,
\cite{BerlaiRobbio:combinationHHG,
HagenSusse:HHScubical,Russell:relHHS}, as well as in several papers 
in progress.

\subsection*{Acknowledgments} The authors thank Mark Hagen and
Alessandro Sisto for lively conversations about hierarchical
hyperbolicity.  The authors were supported in part by NSF grant
DMS-1440140 while at the Mathematical Sciences Research Institute in
Berkeley during Fall 2016 program in Geometric Group Theory.  Abbott
was supported by the NSF RTG award DMS-1502553 and NSF award 
DMS-1803368, Behrstock was supported by NSF award DMS-1710890, and Durham was
supported by NSF RTG award DMS-1045119 and NSF award DMS-1906487.  Behrstock thanks Chris Leininger for an
interesting conversation which led to the formulation of
Question~\ref{conji:equivalence}.  We thank Daniel Berlyne, Ivan Levcovitz, Jacob
Russell, Davide Spriano, and the anonymous referee for helpful 
feedback. We 
thank Anthony Genevois for resolving a question asked in an early 
version of this article.

\section{Background}

We begin by recalling some preliminary notions about metric spaces, maps between them, and group actions.  Given metric spaces $X,Y$, we use $d_X,d_Y$ to denote the distance functions in $X,Y$, respectively.   A map $f\colon X \to Y$ is \emph{$K$--Lipschitz} if there exists a constant $K\geq 1$ such that for every $x,y\in X$, $d_X(x,y)\leq Kd_Y(g(x),g(y))$; it is \emph{$(K,C)$--coarsely Lipschitz} if $d_X(x,y)\leq Kd_X(x,y)+C$.  The map $f$ is a \emph{$(K,C)$--quasi-isometric embedding} if there exist constants $K\geq 1$ and $C\geq 0$ such that for all $x,y\in X$, \[\frac1Kd_X(x,y)-C\leq d_Y(f(x),f(y))\leq Kd_X(x,y)+C.\]  If, in addition, $Y$ is contained in the $C$--neighborhood of $f(X)$, then $f$ is a \emph{$(K,C)$--quasi-isometry}.  For any interval $I\subseteq \mathbb R$, the image of an isometric embedding $I\to X$ is a \emph{geodesic} and the image of a $(K,C)$--quasi-isometric embedding $I\to X$ is a \emph{$(K,C)$--quasigeodesic}.  

If any two points in $X$ can be connected by a $(K,C)$--quasigeodesic, then we say $X$ is a \emph{$(K,C)$--quasigeodesic space}.  If $K=C$, we may simply say that $X$ is a \emph{$K$--quasigeodesic space}.  A subspace $Z\subseteq X$ is \emph{$K$--quasi-convex} if there exists a constant $K\geq 0$ such that  any geodesic in $X$ connecting points in $Z$ is contained in the $K$--neighborhood of $Z$.  For all of the above notions, if the particular constants $K,C$ are not important, we may drop them and simply say, for example, that a map is a quasi-isometry.

Throughout this paper, we will assume that all group actions are by isometries.  The action of a group $G$ on a metric space $X$, which we denote by $G\curvearrowright X$, is \emph{proper} if for every bounded subset $B\subseteq X$, the set $\{g\in G\mid gB\cap B\neq \emptyset\}$ is finite.  The action is  \emph{cobounded} (respectively, \emph{cocompact}) if there exists a bounded (respectively, compact) subset $B\subseteq X$ such that $X=\cup_{g\in G} gB$.  If a group $G$ acts on metric spaces $X$ and $Y$, we say a map $f\colon X\to Y$ is \emph{$G$--equivariant} if for every $x\in X$,  $f(gx)=gf(x)$.  A \emph{quasi-action} of $G$ on $X$ associates to each $g\in G$ a quasi-isometry $A_g\colon x\to gx$ of $X$ with uniform quasi-isometry constants, such that $A_g\circ A_h$ is within uniformly bounded distance of $A_{gh}$.

\subsection{Hierarchically hyperbolic spaces}\label{subsec:defHHS}

In this section we recall the basic definitions and 
properties of 
hierarchically hyperbolic spaces as introduced in  
\cite{BehrstockHagenSisto:HHS_I, BehrstockHagenSisto:HHS_II}.

\begin{defn}[Hierarchically hyperbolic space]\label{defn:HHS}
A $q$--quasigeodesic space  $(\cuco X,\dist_{\cuco X})$ is said to be  
\emph{hierarchically hyperbolic} if there exists 
$\delta\geq0$, an index set $\mathfrak S$, and a set $\{\fontact W \mid W\in\mathfrak S\}$ of $\delta$--hyperbolic spaces $(\fontact U,\dist_U)$,  such that the following conditions are satisfied: \begin{enumerate}
\item\textbf{(Projections.)}\label{item:dfs_curve_complexes} There is
a set $\{\pi_W\co \cuco X\rightarrow2^{\fontact W}\mid W\in\mathfrak S\}$
of \emph{projections} sending points in $\cuco X$ to sets of diameter
bounded by some $\xi\geq0$ in the various $\fontact W\in\mathfrak S$.
Moreover, there exists $K$ so that each $\pi_W$ is $(K,K)$--coarsely
Lipschitz and $\pi_W(\cuco X)$ is $K$--quasiconvex in $\fontact
W$.

 \item \textbf{(Nesting.)} \label{item:dfs_nesting} $\mathfrak S$ is
 equipped with a partial order $\nest$, and either $\mathfrak
 S=\emptyset$ or $\mathfrak S$ contains a unique $\nest$--maximal
 element which is larger than all other elements; when $V\nest W$, we say $V$ is \emph{nested} in $W$.
 For each
 $W\in\mathfrak S$, we denote by $\mathfrak S_W$ the set of
 $V\in\mathfrak S$ such that $V\nest W$.  Moreover, for all $V,W\in\mathfrak S$
 with $V\propnest W$ there is a specified subset
 $\rho^V_W\subset\fontact W$ with $\diam_{\fontact W}(\rho^V_W)\leq\xi$.
 There is also a \emph{projection} $\rho^W_V\colon \fontact
 W\rightarrow 2^{\fontact V}$. 
 
 \item \textbf{(Orthogonality.)} 
 \label{item:dfs_orthogonal} $\mathfrak S$ has a symmetric and
 anti-reflexive relation called \emph{orthogonality}: we write $V\orth
 W$ when $V,W$ are orthogonal.  Also, whenever $V\nest W$ and $W\orth
 U$, we require that $V\orth U$.  Finally, we require that for each
 $T\in\mathfrak S$ and each $U\in\mathfrak S_T$ for which
 $\{V\in\mathfrak S_T\mid V\orth U\}\neq\emptyset$, there exists $W\in
 \mathfrak S_T-\{T\}$, so that whenever $V\orth U$ and $V\nest T$, we
 have $V\nest W$; we say $W$ is a \emph{container} associated with $T\in\mathfrak S$ and $U\in\mathfrak S_T$.  Finally, if $V\orth W$, then $V,W$ are not
 $\nest$--comparable.
 
 \item \textbf{(Transversality and consistency.)}
 \label{item:dfs_transversal} If $V,W\in\mathfrak S$ are not
 orthogonal and neither is nested in the other, then we say $V,W$ are
 \emph{transverse}, denoted $V\transverse W$.  There exists
 $\kappa_0\geq 0$ such that if $V\transverse W$, then there are
  sets $\rho^V_W\subseteq\fontact W$ and
 $\rho^W_V\subseteq\fontact V$ each of diameter at most $\xi$ and 
 satisfying: $$\min\left\{\dist_{
 W}(\pi_W(x),\rho^V_W),\dist_{
 V}(\pi_V(x),\rho^W_V)\right\}\leq\kappa_0$$ for all $x\in\cuco X$.
 
 For $V,W\in\mathfrak S$ satisfying $V\nest W$ and for all
 $x\in\cuco X$, we have: $$\min\left\{\dist_{
 W}(\pi_W(x),\rho^V_W),\diam_{\fontact
 V}(\pi_V(x)\cup\rho^W_V(\pi_W(x)))\right\}\leq\kappa_0.$$ 
  
 Finally, if $U\nest V$, then $\dist_W(\rho^U_W,\rho^V_W)\leq\kappa_0$ whenever $W\in\mathfrak S$ satisfies either $V\propnest W$ or $V\transverse W$ and $W\not\orth U$.
 
 \item \textbf{(Finite complexity.)} \label{item:dfs_complexity} There exists $n\geq0$, the \emph{complexity} of $\cuco X$ (with respect to $\mathfrak S$), so that any set of pairwise--$\nest$--comparable elements has cardinality at most $n$.
  
 \item \textbf{(Large links.)} \label{item:dfs_large_link_lemma} There
exist $\lambda\geq1$ and $E\geq\max\{\xi,\kappa_0\}$ such that the following holds.
Let $W\in\mathfrak S$ and let $x,x'\in\cuco X$.  Let
$N=\lambda\dist_{_W}(\pi_W(x),\pi_W(x'))+\lambda$.  Then there exists $\{T_i\}_{i=1,\dots,\lfloor
N\rfloor}\subseteq\mathfrak S_W-\{W\}$ such that for all $T\in\mathfrak
S_W-\{W\}$, either $T\in\mathfrak S_{T_i}$ for some $i$, or $\dist_{
T}(\pi_T(x),\pi_T(x'))<E$.  Also, $\dist_{
W}(\pi_W(x),\rho^{T_i}_W)\leq N$ for each $i$.

 \item \textbf{(Bounded geodesic image.)}
 \label{item:dfs:bounded_geodesic_image} For all $W\in\mathfrak S$,
 all $V\in\mathfrak S_W-\{W\}$, and all geodesics $\gamma$ of
 $\fontact W$, either $\diam_{\fontact V}(\rho^W_V(\gamma))\leq E$ or
 $\gamma\cap\neb_E(\rho^V_W)\neq\emptyset$.
 
 \item \textbf{(Partial Realization.)} \label{item:dfs_partial_realization} There exists a constant $\alpha$ with the following property. Let $\{V_j\}$ be a family of pairwise orthogonal elements of $\mathfrak S$, and let $p_j\in \pi_{V_j}(\cuco X)\subseteq \fontact V_j$. Then there exists $x\in \cuco X$ so that:
 \begin{itemize}
 \item $\dist_{V_j}(x,p_j)\leq \alpha$ for all $j$,
 \item for each $j$ and 
 each $V\in\mathfrak S$ with $V_j\propnest V$, we have 
 $\dist_{V}(x,\rho^{V_j}_V)\leq\alpha$, and
 \item if $W\transverse V_j$ for some $j$, then $\dist_W(x,\rho^{V_j}_W)\leq\alpha$.
 \end{itemize}

\item\textbf{(Uniqueness.)} For each $\kappa\geq 0$, there exists
$\theta_u=\theta_u(\kappa)$ such that if $x,y\in\cuco X$ and
$\dist(x,y)\geq\theta_u$, then there exists $V\in\mathfrak S$ such
that $\dist_V(x,y)\geq \kappa$.\label{item:dfs_uniqueness}
\end{enumerate}
\end{defn}

\begin{notation}
Note that below we will often abuse notation by simply writing $(\cuco X, \mathfrak S)$ or $\mathfrak S$ to 
refer to the entire package of an hierarchically hyperbolic 
structure, including all the associated spaces, projections, and 
relations given by the above definition.
\end{notation}

\begin{notation}\label{notation:suppress_pi}
When writing distances in $\fontact U$ for some $U\in\mathfrak S$, we 
often simplify the notation slightly by suppressing the projection
map $\pi_U$, i.e., given $x,y\in\cuco X$ and
$p\in\fontact U$  we write
$\dist_U(x,y)$ for $\dist_U(\pi_U(x),\pi_U(y))$ and $\dist_U(x,p)$ for
$\dist_U(\pi_U(x),p)$. Note that when we measure distance between a 
pair of sets (typically both of bounded diameter) we are taking the minimum distance 
between the two sets. For distance/diameter, if 
the space in which the measurement is being made is not clear from 
the context, we will denote it by a subscript.
Given $A\subset \cuco X$ and $U\in\mathfrak S$ 
we let $\pi_{U}(A)$ denote $\cup_{a\in A}\pi_{U}(a)$.
\end{notation}

\begin{rem}
In the setting of hierarchically hyperbolic spaces, we often encounter
maps which are well-defined only up to uniformly bounded error, in the
following sense.  Given a map $f\colon X\to Y$ between quasi-geodesic
spaces $X,Y$, there may be multiple possible points in $Y$ that one
could define as $f(x)$ for a particular $x\in X$.  If the diameter of
such possible points $f(x)$ is uniformly bounded in $Y$ over all $x\in
X$, then we say that the
map is \emph{coarsely well-defined}, since we could 
arbitrarily make a choice for each $f(x)$ and the map would be 
well-defined up to uniformly bounded error.  For example, $\rho^U_V$ gives a
coarsely well-defined map $\fontact U\to\fontact V$.
\end{rem}

An important consequence of being a hierarchically hyperbolic space is the following distance formula, which relates distances in $\X$ to distances in the hyperbolic spaces $\fontact U$ for $U\in\s$.  The notation $\ignore{x}{s}$ means include $x$ in the 
sum if and 
only if $x>s$.

\begin{thm}[Distance formula for HHS;  
    \cite{BehrstockHagenSisto:HHS_II}]\label{thm:distance_formula}
Let $(\cuco X, \mathfrak S)$ be a hierarchically hyperbolic space.  Then
there exists $s_0$ such that for all $s\geq s_0$, there exist $C,K$ so
that for all $x,y\in\cuco X$,
$$\dist(x,y)\asymp_{K,C}\sum_{U\in\mathfrak
S}\ignore{\dist_U(x,y)}{s}.$$
\end{thm}

\medskip

We now recall an important construction of subspaces in a hierarchically hyperbolic space called 
\emph{standard product regions} introduced in 
\cite[Section~13]{BehrstockHagenSisto:HHS_I} and studied further in 
\cite{BehrstockHagenSisto:HHS_II}. First we define a \emph{consistent tuple}, which will be used to define the two factors in 
the product space.

\begin{defn}[Consistent tuple] \label{defn:consistent_tuple}
Fix $\kappa\geq 0$, and let $\tup{b}\in\prod_{U\in\s} 2^{\fontact U}$ be a tuple such that for each $U\in\s$, the coordinate $b_U$ is a subset of $\fontact U$ with $\diam_{\fontact U}(b_U)\leq \kappa$.  The tuple $\tup{b}$ is \emph{$\kappa$--admissible} if $\dist_U(b_U,\pi_U(\X))\leq \kappa$ for all $U\in\s$.  The $\kappa$--admissible tuple $\tup{b}$ is \emph{$\kappa$--consistent} if, whenever $V\trans W$,
\[\min\left\{\dist_W(b_W,\rho^V_W),\dist_V(b_V,\rho^W_V)\right\}\leq \kappa \]
and whenever $V\nest W$,
\[\min\left\{\dist_W(b_W,\rho^V_W),\diam_{\fontact V}(b_V\cup\rho^W_V(b_W))\right\}\leq \kappa.\]
\end{defn}

\begin{defn}[Nested partial tuple ($\FU U$)]\label{defn:nested_partial_tuple}
Recall $\mathfrak S_U=\{V\in\mathfrak S \mid V\nest U\}$.  Fix
$\kappa\geq\kappa_0$ and let $\FU U$ be the set of
$\kappa$--consistent tuples 
in $\prod_{V\in\mathfrak S_U}2^{\fontact
V}$.
\end{defn}

\begin{defn}[Orthogonal partial tuple ($\mathbf E_U$) ]\label{defn:orthogonal_partial_tuple}
Let $\mathfrak S_U^\orth=\{V\in\mathfrak S\mid V\orth U\}\cup\{A\}$, where
$A$ is a $\nest$--minimal element $W$ such that $V\nest W$ for all
$V\orth U$ (note that $A$ exists by the container axiom for an HHS, 
i.e., Definition~\ref{defn:HHS}.(\ref{item:dfs_orthogonal})).  Fix $\kappa\geq\kappa_0$, let $\mathbf E_U$ be the set of
$\kappa$--consistent tuples in $\prod_{V\in\mathfrak
S_U^\orth-\{A\}}2^{\fontact V}$.
\end{defn}

\begin{defn}[Product regions in $\cuco X$]\label{const:embedding_product_regions}
Given $\cuco X$ and $U\in\mathfrak S$, there is a coarsely
well-defined map $\phi_U\co\mathbf F_{U}\times \mathbf E_U\to\cuco X$ which restricts to coarsely well-defined maps $\phi^\nest,\phi^\orth\co\FU U,\mathbf E_U\to\cuco X$. 
Indeed, for each $(\tup a,\tup b)\in
\mathbf F_U\times \mathbf E_U$, and each $V\in\mathfrak S$, the
projection $\pi_V(\phi_U(\tup a,\tup b))$ is defined as follows.  If $V\nest U$,
then $\pi_V(\phi_U(\tup a,\tup b))=a_V$.  If $V\orth U$, then
$\pi_V(\phi_U(\tup a,\tup b))=b_V$.  If $V\transverse U$, then
$\pi_V(\phi_U(\tup a,\tup b))=\rho^U_V$.  Finally, if $U\nest V$, and
$U\neq V$, let $\pi_V(\phi_U(\tup a,\tup b))=\rho^U_V$.  The tuple $(\pi_V(\phi_U(\tup a,\tup b)))_{V\in \s}\in\prod_{V\in \s} 2^{\fontact V}$ is $\kappa$--consistent (see
\cite[Construction~5.10]{BehrstockHagenSisto:HHS_II}), and therefore
\cite[Theorem~3.1]{BehrstockHagenSisto:HHS_II} provides a point
$x\in\X$ such that $\dist_W(\pi_W(x),\pi_W(\phi_U(\tup a,\tup b)))\leq
\theta_e$ for all $W\in \s$.  Moreover, the point $x$ is \emph{coarsely
unique} in the sense that the set of all $x$ which satisfy
$\dist_W(\pi_W(x),\pi_W(\phi_U(\tup a,\tup b)))\leq \theta_e$ for each
$W\in\s$ has diameter at most $\theta_u$ in $\X$.  We define $\phi_U(\tup
a,\tup b)=x$; the coarse uniqueness of $x$ shows that this map is
coarsely well-defined.  Fixing any $e\in \mathbf E_U$ yields a map $\phi_U^{\nest}\colon \mathbf F_U\times \{e\}\to \X$, and $\phi^\perp$ is defined analogously.  We refer to
$\FU U\times \EU U$ as a \emph{product region}, which we denote $\mathbf P_{U}$.
\end{defn}

We often abuse notation slightly and use the notation $\EU U,\FU U$, and 
$\mathbf P_{U}$ to refer to the image in $\cuco X$ of the associated 
set.
In \cite[Construction~5.10]{BehrstockHagenSisto:HHS_II} it is proven that 
these standard product regions have the property that they are 
``hierarchically quasiconvex subsets'' of $\cuco X$.   We leave 
out the definition of hierarchically quasiconvexity, because its 
only use here is that product regions have ``gate maps,'' as given by the 
following in \cite[Lemma~5.5]{BehrstockHagenSisto:HHS_II}:

\begin{lem}[Existence of coarse gates; {\cite[Lemma~5.5]{BehrstockHagenSisto:HHS_II}}]
	\label{lem:gate}
If $\cuco Y\subseteq\cuco X$ is $k$--hierarchically quasiconvex and
non-empty, then there exists a gate map for $\cuco Y$, i.e., for each
$x\in \cuco X$ there exists $\mathfrak g (x) \in\cuco Y$ such that for all
$V\in\mathfrak S$, the set $\pi_V(\mathfrak g (x))$ (uniformly) coarsely coincides
with the projection of $\pi_V(x)$ to the $k(0)$--quasiconvex set
$\pi_V(\cuco Y)$. The point $\mathfrak g (x)\in \cuco Y$ is called 
the \emph{gate of $x$ in $\cuco Y$}.
\end{lem}

\begin{rem}[Surjectivity of projections]\label{rem:surjectivity}
    As one can always
    change the hierarchical structure so that the projection maps are
    coarsely surjective \cite[Remark~1.3]{BehrstockHagenSisto:HHS_II}, 
    we will assume that $\s$ is such a structure.
    That is, for each $U\in \s$, if $\pi_{U}$ is not surjective, then 
    we identify $\fontact U $  with $\pi_U(\cuco X)$. 
\end{rem}

We also need the notion of a hierarchy path, whose existence was
proven in \cite[Theorem 4.4]{BehrstockHagenSisto:HHS_II} (although we 
use the word \emph{path}, since they are quasi-geodesics, typically 
we consider them as discrete sequences of points):

\begin{defn}
For $R\geq 1$, a path $\gamma$ in $\cuco X$ is a \emph{$R$--hierarchy path} if
 \begin{enumerate}
  \item $\gamma$ is a $(R,R)$--quasigeodesic,
  \item for each $W\in\mathfrak S$, $\pi_W\circ\gamma$ is an unparametrized $(R,R)$--quasigeodesic.
An unbounded hierarchy path $[0,\infty)\to\cuco X$ is a \emph{hierarchy ray}.
\end{enumerate}
\end{defn}

We call a domain \emph{relevant} to a pair of points, if the 
projections to that domain are larger than some fixed (although 
possibly unspecified) constant depending only on the hierarchically 
hyperbolic structure. We say a domain is \emph{relevant} for a 
particular quasi-geodesic if it is relevant for the endpoints of that 
quasi-geodesic.

\begin{prop}[{\cite[Proposition~5.17]{BehrstockHagenSisto:HHS_II}}]\label{prop:HHSII5.15}
There exists $\nu\geq 0$ such that for all $x,y\in \cuco X$, all $V\in \s$ with $V$ relevant for $(x,y)$, and all $D$--hierarchy paths $\gamma$ joining $x$ to $y$, there is a subpath $\alpha$ of $\gamma$ with the following properties:
	\begin{enumerate}
	\item $\alpha\subset \mathcal N_\nu(\mathbf P_V)$;
	\item $\pi_U\vert_\gamma$ is coarsely constant on $\gamma-\alpha$
	for all $U\in\s_V\cup\s_V^\perp$, i.e., it is a uniformly bounded 
	distance from a constant map.
	\end{enumerate}
\end{prop}

\begin{rem}\label{rem:activeinproductregion}
Let $x,y\in \cuco X$, and suppose $V$ is relevant for $(x,y)$.  As
$\FU V$ and $\EU V$ consist of $\kappa$--consistent tuples (for a 
fixed $\kappa$) and
$\phi_V\colon \FU V\times \EU V\to\cuco X$ is only coarsely
well-defined, by appropriately increasing $\kappa$ to accomodate for 
the chosen constant $\nu$ in Proposition~\ref{prop:HHSII5.15}, 
we may assume that
$\alpha$ is actually a subset of $\mathbf P_V$.
\end{rem}

It is often convenient to work with equivariant hierarchically
hyperbolic structures, we now recall the relevant structures for 
doing so.  For details see \cite{BehrstockHagenSisto:HHS_II}.

\begin{defn}[Hierarchically hyperbolic groups] \label{def:HHG}
    Let $(\cuco X,\mathfrak S)$ be a hierarchically hyperbolic 
    space.  
     An \emph{automorphism} of $(\cuco X,\mathfrak S)$ consists of a map
    $g\colon \cuco X\to\cuco X$, together with a bijection
    $g\inducedS\colon \mathfrak S\to\mathfrak S$ and, for each
    $U\in\mathfrak S$, an isometry $g\induced(U)\colon \fontact
    U\to\fontact {(g\inducedS{(U)})}$ so that the following diagrams commute up to uniformly bounded error
    whenever the maps in question are defined (i.e., when $U,V$ are
    not orthogonal):
    \begin{center}
    $
    \begin{diagram}
    \node{\cuco 
    X}\arrow[3]{e,t}{g}\arrow{s,l}{\pi_U}\node{}\node{}\node{\cuco 
    X'}\arrow{s,r}{\pi_{g\inducedS(U)}}\\
    \node{\fontact U}\arrow[3]{e,t}{g\induced(U)}\node{}\node{}\node{\fontact 
    (g\inducedS(U))}
    \end{diagram}
    $
    \end{center}
    and
    \begin{center}
    $
    \begin{diagram}
    \node{\fontact 
    U}\arrow[4]{s,l}{\rho^U_V}\arrow[7]{e,t}{g\induced(U)}\node{}\node{}\node{}\node{}\node{}\node{}\node{}\node{}\node{}\node{\fontact (g\inducedS(U))}\arrow[4]{s,r}{\rho^{g\inducedS(U)}_{g\inducedS(V)}}\\
    \node{}\node{}\node{}\node{}\node{}\node{}\node{}\node{}\\
    \node{}\node{}\node{}\node{}\node{}\node{}\node{}\node{}\\
    \node{}\node{}\node{}\node{}\node{}\node{}\node{}\node{}\\
    \node{\fontact 
    V}\arrow[7]{e,t}{g\induced(V)}\node{}\node{}\node{}\node{}\node{}\node{}\node{}\node{}\node{}\node{\fontact (g\inducedS (V))}
    \end{diagram}
    $
    \end{center}
	Two automorphisms $f,f'$ are \emph{equivalent} if
	$f^\diamond=(f')^\diamond$ and for all $U\in \s$ we have $\phi_U=\phi'_U$.
	The set of all such equivalence classes forms the
	\emph{automorphism group} of $(\cuco X,\mathfrak S)$, denoted
	$\Aut(\cuco X,\mathfrak S)$.  A finitely generated group $G$ is
	said to be a \emph{hierarchically hyperbolic group (HHG)} if there
	is a hierarchically hyperbolic space $(\cuco X,\mathfrak S)$ and a
	group homomorphism $G\to\Aut(\cuco X,\frak S)$ so that the induced
	uniform quasi-action of $G$ on $\cuco X$ is metrically proper,
	cobounded, and $\frakS$ contains finitely many $G$--orbits.  Note
	that when $G$ is a hyperbolic group then, with respect to any word
	metric, it inherits a hierarchically hyperbolic structure.
\end{defn}

\subsection{Acylindrical actions}

We recall the basic definitions related to acylindrical actions; the 
canonical references are \cite{Bowditch:tight} and  \cite{Osin:acyl}. 
We also discuss a partial order on these actions which was recently introduced in  
\cite{AbbottBalasubramanyaOsin:extensions}.

\begin{defn}[Acylindrical]
The action of a group $G$ on a metric space $X$ is 
\emph{acylindrical} if for any $\varepsilon>0$ there exist $R,N\geq 0$ such that for all $x,y\in X$ with $\dist(x,y)\geq R$, 
\[
|\{g\in G\mid \dist (x,gx)\leq \varepsilon \textrm{ and } \dist (y,gy)\leq \varepsilon \}|\leq N.
\]
\end{defn}

Recall that given a group $G$ acting on a hyperbolic metric space $X$, an element $g\in G$ is \emph{loxodromic} if the map $\mathbb Z\to X$ defined by $n\mapsto g^nx$ is a quasi-isometric embedding for some (equivalently any) $x\in X$.  However, an element of $G$ may be loxodromic for some actions and not for others.  
Consider, for example, the free group on two generators acting on its Cayley graph and acting on the Bass-Serre tree associated to the splitting $\mathbb F_2\simeq \langle x\rangle *\langle y\rangle$.  In the former action, every non-trivial element is loxodromic, while in the latter action, no powers of $x$ and $y$ are loxodromic.

\begin{defn}[Generalized loxodromic]  An element of a group $G$ is called \emph{generalized loxodromic} if it is loxodromic for some acylindrical action of $G$ on a hyperbolic space.  
\end{defn}

\begin{rem} \label{products} By \cite[Theorem 1.1]{Osin:acyl}, every
acylindrical action of a group on a hyperbolic space either has
bounded orbits or contains a loxodromic element.  By \cite[Theorem
1.4.(L4)]{Osin:acyl} and Sisto \cite[Theorem 1]{Sisto:qconvex}, every
generalized loxodromic element is \emph{Morse}, i.e., every 
quasi-geodesic with endpoints on the axis of the element lies 
uniformly close to that axis (see Definition \ref{defn:Morse}).  Therefore, if a group $H$
does not contain any Morse elements, it does not contain any
generalized loxodromics, and thus $H$ must have bounded orbits in
every acylindrical action on a hyperbolic space.  This is the case
when, for example, $H$ is a non-trivial direct product, that is, a direct product of two infinite groups.
\end{rem}

\begin{defn}[Universal acylindrical action] An acylindrical action of a group on a hyperbolic space is a {\it universal acylindrical action} if every generalized loxodromic element is loxodromic.  Such an action is sometimes called a {\it loxodromically universal action}.
\end{defn}

Notice that if every acylindrical action of a group $G$ on a hyperbolic space has bounded orbits, then $G$ does not contain any generalized loxodromic elements, and the action of $G$ on a point (which is acylindrical) is a universal acylindrical action.

The following notions are discussed in detail in \cite{AbbottBalasubramanyaOsin:extensions}.  We give a brief overview here.
 Fix a group $G$.  Given a (possibly infinite) generating set $X$ of $G$, let $|\cdot|_X$
 denote the word metric with respect to $X$, and let $\Gamma(G,X)$ 
be the Cayley graph of $\Gamma$ with respect to the generating set $X$.  Given two generating
 sets $X$ and $Y$, we say $X$ is \emph{dominated} by $Y$ and write
 $X\preceq Y$ if
\[\sup_{y\in Y}|y|_X<\infty.\] Note that when 
$X\preceq Y$, the action $G\curvearrowright
\Gamma(G,Y)$ provides more information about the group than
$G\curvearrowright\Gamma(G,X)$, and so, in some sense, is a ``larger''
action.
The two generating sets $X$ and $Y$ are equivalent if $X\preceq Y$ and
$Y\preceq X$; when this happens we write $X\sim Y$.

Let $\mathcal{AH}(G)$ be the set of equivalence classes of generating
sets $X$ of $G$ such that $\Gamma(G,X)$ is hyperbolic and the action
$G\curvearrowright \Gamma(G,X)$ is acylindrical.   We denote the equivalence class of $X$ by $[X]$.  The preorder
$\preceq$ induces an order on $\mathcal{AH}(G)$, which we also denote 
$\preceq$.  

\begin{defn}[Largest] We say an equivalence class of generating sets is \emph{largest} if it is the largest
element in $\mathcal{AH}(G)$ under this ordering.
\end{defn}

Given a cobounded acylindrical action of $G$ on a hyperbolic space
$S$, a Milnor--Schwartz argument gives a (possibly infinite)
generating set $Y$ of $G$ such that there is a $G$--equivariant
quasi-isometry between $G\curvearrowright S$ and $G\curvearrowright
\Gamma(G,Y)$.
By a slight abuse of language, we will say that a particular cobounded acylindrical action $G\curvearrowright S$ on a hyperbolic space is largest, when, more precisely, it is the equivalence class of the generating set  associated to this action through the above correspondence, $[Y]$, that is the largest element in $\mathcal{AH}(G)$.

\begin{rem}
By definition, every largest acylindrical action is a universal acylindrical action.  To see this, notice that if $[X]\preceq[Y]$, then the set of loxodromic elements in $G\curvearrowright \Gamma(G,X)$ must be a subset of the set of loxodromic elements in $G\curvearrowright \Gamma(G,Y)$.
\end{rem}

\subsection{Stability} \label{subsec:stability}

Stability is strong coarse convexity property which generalizes quasiconvexity in hyperbolic spaces and convex cocompactness in mapping class groups \cite{DurhamTaylor:stable}.  In the general context of metric spaces, it is essentially the familiar Morse property generalized to subspaces, so we begin there.

\begin{defn}[Morse/stable quasigeodesic]\label{defn:Morse}
Let $X$ be a metric space.  A quasigeodesic $\gamma \subset X$ is
called \emph{Morse} (or $\emph{stable}$) if there exists
a function $N\co \mathbb{R}^2_{\geq 0} \rightarrow \mathbb{R}_{\geq 0}$
such that if $q$ is a $(K,C)$--quasigeodesic in $X$ with endpoints on
$\gamma$, then
$$q \subset \mathcal{N}_{N(K,C)}(\gamma).$$
We call $N$ the \emph{stability gauge} for $\gamma$ and say $\gamma$ 
is $N$--stable if we want to record the constants.
\end{defn}

We can now define a notion of stable embedding of one metric space in another which is equivalent to the one introduced by Durham and Taylor \cite{DurhamTaylor:stable}:

\begin{defn}[Stable embedding]\label{defn:stable embedding}
We say a quasi-isometric embedding $f\co X \rightarrow Y$ between
quasigeodesic metric spaces is a \emph{stable embedding} if there
exists a stability gauge $N$ such that for any quasigeodesic constants
$K,C$ and any 
$(K,C)$--quasigeodesic $\gamma \subset X$, then $f(\gamma)$ is an  
$N(K,C)$--stable
quasigeodesic in $Y$.  We say a subset $X\subseteq Y$ is \emph{stable} if it is undistorted and the inclusion map $i\colon X\to Y$ is a stable embedding.
\end{defn}

The following generalizes the notion of a Morse 
quasigeodesic to subgroups:

\begin{defn}[Subgroup stability]
Let $H$ be subgroup of a finitely generated group $G$.  We say $H$ is a
\emph{stable subgroup} of $G$ if some (equivalently, any) orbit map of
$H$ into some (any) Cayley graph (with respect to a finite generating 
set) of $G$ is a stable embedding.

If for some $h \in G$, $H = \langle h \rangle$  is stable, then we call 
$h$  \emph{stable}. Such elements are often called \emph{Morse elements}.
\end{defn}

Stability of a subset is preserved under quasi-isometries.  Note that stable
subgroups are undistorted in their ambient groups and,
moreover, they are quasiconvex with respect to any choice of
finite generating set for the ambient group.

\section{Altering the hierarchically hyperbolic structure}\label{sec:structures}

The goal of this section is to prove that any hierarchically
hyperbolic space satisfying a technical assumption---the
\emph{bounded domain dichotomy}---admits a hierarchically hyperbolic structure with
unbounded products, i.e., every non-trivial product region in the ambient space 
has unbounded factors; see Theorem
\ref{thm:betterHHSstructure} below.

In particular, this establishes that all 
hierarchically hyperbolic groups admit a hierarchically hyperbolic
group structure with unbounded products. It is for this reason that 
our complete characterization of the contracting property in 
spaces with unbounded products in Section \ref{sec:Charact_contractions} 
yields a characterization of the contracting property for 
all hierarchically hyperbolic
groups, as stated in Theorem~\ref{thmi:char of contracting}.

\subsection{Unbounded products}

Fix a hierarchically hyperbolic space $(\cuco X, \mathfrak S)$.

Let $M>0$ and let $\s^M \subset \mathfrak S$ be the set of
domains $U \in \mathfrak S$ such that there exists $V \in \mathfrak 
S$ and $W \in \mathfrak S_V^{\orth}$ satisfying: 
$U \nest V$, 
$\diam(\fontact V)>M$, and 
$\diam(\fontact W) > M$.

Recall that a set of domains $\mathfrak U \subset \mathfrak S$ is
\emph{closed under nesting} if whenever $U \in \mathfrak U$ and $V \nest U$,
then $V \in \mathfrak U$.

\begin{lem}\label{lem:closed under nesting}
For any $M>0$, the set $\s^M$ is closed under nesting.
\end{lem}

\begin{proof}
Let $U \in \s^M$ and $V \nest U$. 
By definition of $U\in\s^{M}$, there exists $Z \in \s^M$ with 
$U \nest Z$ and satisfying: 
$\diam(\fontact Z)>M$ and there exists 
$W \in \mathfrak S_Z^{\orth}$ such that $\diam(\fontact W) > M$. 
Since $V \nest Z$, it follows that $V\in\s^{M}$,
as desired.
\end{proof}

\begin{defn}[Bounded domain dichotomy]\label{defn:bounded dd}
We say $(\cuco X, \mathfrak S)$ has the \emph{$M$--bounded
domain dichotomy} if there exists $M>0$ such that any $U \in 
\mathfrak S$ with $\diam(\fontact U) > M$ satisfies $\diam(\fontact U) = \infty$.
If the value of $M$ is not important, we simply refer to the 
\emph{bounded domain dichotomy}.
\end{defn}

Recall that for every hierarchically hyperbolic group $(G,\s)$, the set of domains $\s$ contains  finitely many $G$--orbits and each $g\in G$ induces an isometry $\fontact U\to \fontact (g^\diamond(U))$ for each $U\in\s$ (see Definition \ref{def:HHG}).  It thus follows that every hierarchically hyperbolic group has the bounded 
domain dichotomy. (Also, note that this property implies the space is 
``asymphoric'' as defined in
\cite{BehrstockHagenSisto:quasiflats}.)

\begin{defn}[Unbounded products]\label{defn:unbounded products}
We say that a hierarchically hyperbolic space $(\cuco X, \mathfrak S)$
has \emph{unbounded
products} if it has the bounded domain dichotomy and the property that 
if $U \in \mathfrak S - \{S\}$ has $\diam(\mathbf F_U) = \infty$, 
then 
$\diam(\mathbf E_U) = \infty$. 
\end{defn}

\subsection{Almost hierarchically hyperbolic spaces}\label{subsec:almost HHS}

In this section we introduce a tool for verifying a space is 
hierarchically hyperbolic.

The following is a weaker version of the orthogonality axiom:

\begin{enumerate}
\item[$(3')$] {\bf (Bounded pairwise orthogonality)} $\frakS$ has a symmetric and anti-reflexive relation called \emph{orthogonality}: we write $V\perp W$ when $V,W$ are orthogonal.  Also, whenever $V\nest W$ and $W \perp U$, we require that $V\perp W$.  Moreover, if $V\perp W$, then $V,W$ are not $\nest$--comparable.
  Finally, the cardinality of any collection of pairwise orthogonal domains is uniformly bounded by $\xi$.
  \end{enumerate}

By \cite[Lemma 2.1]{BehrstockHagenSisto:HHS_II}, the orthogonality
axiom (Definition \ref{defn:HHS}, (3)) for an hierarchically
hyperbolic structure implies axiom $(3')$.  However, the converse does
not hold; that is, the last condition of $(3')$ does not directly imply the
container statement in (3), and thus this is an {\it a priori} strictly weaker
assumption.  However, as is proven in the appendix in Theorem~\ref{thm:almost_HHS_are_HHS}, this weakened 
version of the axiom is enough to 
produce a hierarchically hyperbolic structre.

We now introduce the notion of an \emph{almost hierarchically hyperbolic
space}:

\begin{defn}[Almost HHS] \label{almostHHS}
If $(\X,\s)$ satisfies all axioms of a hierarchically hyperbolic space
except (3) and additionally satisfies axiom $(3')$, then
$(\X,\s)$ is an \emph{almost hierarchically hyperbolic space}.
\end{defn}

In the appendix, Berlyne and Russell prove 
Theorem~\ref{thm:almost_HHS_are_HHS}, establishing that if a 
space is 
almost hierarchically hyperbolic, then the associated structure can 
be modified to obtain a hierarchically hyperbolic structure on the 
original space. This result is used in our proof of 
Theorem~\ref{thm:betterHHSstructure}.

\subsection{A new hierarchically hyperbolic structure} \label{sec:newHHS}

In this section we describe a new hierarchically hyperbolic structure
on hierarchically hyperbolic spaces with the bounded domain dichotomy. 
We first describe the hyperbolic spaces that
will be part of the new structure.

Let $(\cuco X, \mathfrak S)$ be a hierarchically hyperbolic space with
the $M$--bounded domain dichotomy.  Recall that we define $\mathfrak S^{M} \subset
\mathfrak S$ to be the set of $U\in\mathfrak S$ such that there 
exists $U\nest V$ with $\diam(\fontact V)>M$ 
for which there exists a 
$W \in \mathfrak S_V^{\orth}$ satisfying $\diam(\fontact W) > M$.  For each $U \in
\mathfrak S$, define $\mathfrak S^M_{U} \subset \mathfrak S_U$
similarly.

\begin{rem}[Factored spaces]\label{rem:factoredspaces}
	As defined in \cite{BehrstockHagenSisto:asdim}, given $(\cuco
	X,\frakS)$ and $\frakT \subset \frakS$ the \emph{factored space}
	$\wideFU {\frakT}$ is the space obtained from $\cuco X$ by
	coning-off each $\FU V\times \{e\}$ for all $V\in\frakT$ and all $e\in\EU V$.  Sometimes we abuse
	language slightly and refer to this as the factored space obtained
	from $\cuco X$ by collapsing $\frakT$.  In particular, when $S$ is
	the $\nest$--maximal element of $\frakS$, then $\fontact S$ can 
	be taken to be the space $\wideFU {\frakS - \{S\}}$,
	which is obtained from $\X$ by coning-off $\mathbf F_U\times \{e\}$ for all
	$U\in\s - \{S\}$ and all $e\in \EU U$.
    
    We often consider the case of a fixed  $(\cuco X,\frakS)$ and 
    $U\in\frakS$ and then 
    applying this construction to the hierarchy hyperbolic structure   
    $(\FU U,\s_U)$. For this application, 
    note that 
    $\pi_U(\cuco X)$ is quasi-isometric to $\widehat{\mathbf
    F}_{\s_U - \{U\}}$, by 
    \cite[Corollary~2.9]{BehrstockHagenSisto:asdim}, and thus so is $\fontact U$, by
    Remark~\ref{rem:surjectivity}.
\end{rem}

\begin{lem}\label{lemma:coningpaths} Let $(\cuco X,\frakS)$ be a hierarchically hyperbolic 
	space and consider $\frakT\subset \frakS$ which is closed 
	under nesting. Let $\gamma$ be a hierarchy path in $(\cuco 
	X,\frakS)$. Then, the path obtained by including $\gamma\subset 
	\cuco X \subset \wideFU {\frakT}$ is an unparametrized 
	quasi-geodesic. Moreover, if for each $W\in\frak T$ which is a 
	relevant domain for $\gamma$ and for each $e\in \EU W$, we modify the path through $\FU W\times\{e\}$ 
	by removing all but the first and last vertex of the hierarchy 
	path which passes through $\FU W\times\{e\}$, 
	then the new path obtained, $\hat{\gamma}$ is a hierarchy path for 
	$(\wideFU {\frakT},\frakS - \frakT)$.
\end{lem}

\begin{proof}
	The proof is by induction on complexity.  Consider all the
	nest-minimal elements $\mathfrak {U}\subset \frakT$ which are
	relevant for $\gamma$; by Proposition~\ref{prop:HHSII5.15} and 
	Remark~\ref{rem:activeinproductregion} for each such $U$ there is 
	a subpath of $\gamma$ which passes through a collection of slices 
	$\FU U\times\{e\}$ within the product region associated to $U$.  By
	\cite[Lemma~2.14]{BehrstockHagenSisto:HHS_II} there is a bounded
	(in terms of $\mathfrak S$) coloring of $\mathfrak {U}$ with the
	property that all the domains of a given color are pairwise
	transverse.  Starting from $(\cuco X, \frakS)$, we take one color
	at a time, together with all the domains nested inside domains of
	that color, and create the factored space by coning off those
	domains.  At each step, we obtain a new hierarchically hyperbolic
	space with the property that in this space the relevant domains
	for $\gamma$ are exactly the original ones except for those in the
	colors we have coned off thus far.  Since this path still travels
	monotonically through each of the relevant domains, it is an
	unparametrized quasi-geodesic in the new factored space.  Thus the
	path $\hat{\gamma}$ is a parametrized quasi-geodesic 
	and thus a
	hierarchy path in the new factored space (with constants 
	depending only on the constants for the original hierarchy path).  Once the colors of
	$\mathfrak {U}$ are exhausted, repeat one step up the nesting
	lattice.  Since the complexity of a hierarchically hyperbolic and
	the coloring are both bounded, this will terminate after finitely
	many steps.  Finally we cone off any domains in $\frak T$ which
	are not relevant for $\gamma$ to obtain the space $(\wideFU
	{\frakT},\frakS - \frakT)$.  Through this final step
	$\hat{\gamma}$ remains a uniform quality hierarchy path since it is still a
	quasigeodesic.
\end{proof}

The next result uses the above spaces to obtain a hierarchically hyperbolic 
structure with particularly nice properties from a given hierarchically hyperbolic structure.

\begin{thm} \label{thm:betterHHSstructure}
Every hierarchically hyperbolic space with the bounded domain
dichotomy admits a hierarchically hyperbolic
structure with unbounded products.
\end{thm}

\begin{proof}
Let $(\cuco X, \mathfrak S)$ be a hierarchically hyperbolic space. 
Let $\mathfrak T$ denote the $\nest$--maximal element $S$ together with the subset of $\mathfrak S$ consisting of 
all $U\in\mathfrak S$ with
both $\mathbf{F}_U$ and $\mathbf{E}_U$ unbounded.

We begin to define our new hierarchically hyperbolic structure on
$\cuco X$ by taking $\mathfrak T$ as our index set.  For each $U\in
\frakT - \{S\}$ we set the associated hyperbolic space 
$\mathcal T_{U}$ to be
$\fontact U$.  For the top-level domain, $S$, we obtain a hyperbolic 
space, $\mathcal T_{S}$, as follows.  By
Lemma \ref{lem:closed under nesting}, $\mathfrak S^{M}$ is closed
under nesting and hence $\widehat{\cuco X}_{\mathfrak S^{M}}$ is a
hierarchically hyperbolic space.  Moreover, since this hierarchically
hyperbolic space has the property that no pair of orthogonal domains
both have diameter larger than $M$, by
\cite[Corollary~2.16]{BehrstockHagenSisto:quasiflats} it is hyperbolic
for some constant depending only on $(\cuco X, \mathfrak S)$ and $M$; 
we call this space $\mathcal T_{S}$.

To avoid confusion, 
we use the notation $\dist_S$ for distance in $\mathcal T_S$ and the
notation $\dist_{\fontact S}$ for distance in $\fontact S$.

When $U\neq S$, the projections are as defined in the original
hierarchically hyperbolic space.  We take the projection $\pi_S$ to be
the factor map $\cuco X\to\mathcal T_S$.  If $U \in \frakT$ and
$U\neq S$, then the relative projections are defined as in
$(\cuco X, \frakS)$. For the remaining cases the relative projections 
are as follows: $\rho^S_V$ is defined to be $\pi_{V}$ and 
$\rho^V_S$ is defined to be the image of $\mathbf F_V$ under the factor map
$\X\to \mathcal T_S$.

We now check the axioms to verify that $(\X,\frakT)$ is an almost
hierarchically hyperbolic space (i.e., all the conditions of a
hierarchically hyperbolic space except for a weakened version of the
orthogonality axiom).  Once these axioms have been verified, we can then invoke
Theorem~\ref{thm:almost_HHS_are_HHS} to conclude that the
almost hierarchically hyperbolic structure $\frakT$ can be modified to yield
an actual hierarchically hyperbolic space.  By construction,
$(\X,\frakT)$ satisfies the hypothesis of
Corollary~\ref{cor:unbddproducts}, and therefore the associated modified
hierarchically hyperbolic structure will have unbounded products, as
desired.

{\bf Projections:} The only case to check is for the top-level domain 
$S$. Since $\pi_S$ is a factor map, it is coarsely Lipschitz and 
coarsely surjective. 

 {\bf Nesting:} The partial order and projections are given by
 construction.  The diameter bound in the case of nesting projections
 is immediate from the bound from $(\X,\s)$, except in the case of 
 $\rho^{V}_{S}$ for $V\in\frakT$. The bound on the diameter of 
 $\rho^{V}_{S}$ follows from the construction of $\mathcal T_{S}$ as a 
 factor space and the fact that $\frakT \subset \frakS^{M}$.

 {\bf Orthogonality:} We now verify axiom $(3')$ is satisfied by this new structure.
Tthe first three conditions are clear, since $\frakT\subseteq
\s$ and thus they are inherited from the
 hierarchically hyperbolic structure $(\X,\s)$. 
 For the last condition, any collection of pairwise orthogonal
domains in $\frakT$ is also a collections of pairwise orthogonal
domains in $\s$, and thus by
\cite[Lemma~2.2]{BehrstockHagenSisto:HHS_II} has uniformly bounded
size, verifying the axiom.

 {\bf Transversality and consistency:} This axiom only involves 
 domains which are not nest-maximal, and hence holds 
 using the 
 original constants from the
 hierarchically hyperbolic structure on $(\cuco X, \mathfrak S)$.

{\bf Partial realization:} 
This axiom only involves 
 domains which are not nest-maximal, and hence holds 
 using the 
 original constants from the
 hierarchically hyperbolic structure on $(\cuco X, \mathfrak S)$.  
 
{\bf Finite complexity:} This clearly holds by construction.
 
 {\bf Large link axiom:} Let $\lambda$ and $E$ be the constants from the large link axiom for $(\X,\s)$, let $W\in\frakT$, and let $x,x'\in\X$.  Consider the set $\{T_i\}\subset \s_W-\{W\}$ provided by the large link axiom for $(\X,\s)$.  Since $T_i\nest W$, it follows that $\mathbf E_{T_i}$ is unbounded for each $i$.  Let $T\in\frakT_W-\{W\}$.  If $\dist_T(x,x')>E\cdot M$, it follows that $\mathbf F_T$ is unbounded.  Furthermore, $\dist_{\fontact T}(x,x')>E$, whence $T\nest T_i$ for some $i$ by the large link axiom for $(\X,\s)$.  Therefore $\mathbf F_{T_i}$ is unbounded, and so $T_i\in\frakT$. The result follows.

  {\bf Bounded geodesic image:} For all domains in
  $\frakT - \{S\}$, the corresponding hyperbolic spaces are
  unchanged from those in the original structure and thus the axiom 
  holds in these cases.  
  
  Hence the only case which it remains
  to check is when $W=S$.  Suppose $\gamma$ is a geodesic in $\mathcal T_S$, 
  and $V\in \frakT - \{S\}$ such that 
  $\diam_{\fontact V}(\rho^S_V(\gamma))> E$. 
  The partial realization axiom implies that there exists a hierarchy path 
  $\bar{\gamma}\subset \cuco X$ whose end-points project under 
  $\pi_{S}$ to the 
  end-points of $\gamma$. This projected path is a quasigeodesic by 
  Lemma~\ref{lemma:coningpaths}. Since $\mathcal T_S$ is hyperbolic, 
  the projected path lies uniformly close to $\gamma$.
  By \cite[Proposition 5.17]{BehrstockHagenSisto:HHS_II}) we can 
  replace $\bar{\gamma}$ by an appropriate subpath for which 
  the only relevant domains are all nested in $V$; thus  
  $\bar{\gamma}\subset \mathbf P_{V}$. By definition, there is a 
  bounded distance between 
  $\rho^{V}_{S}$ and $\pi_{S}(\mathbf P_{V})$; thus 
  $\pi_{S}(\bar{\gamma})$ (and hence $\gamma$) 
  is a bounded distance from $\rho^{V}_{S}$, as needed. 
  
{\bf Uniqueness:} Let $\kappa > 0$.  We can take $\theta'_u>
\max\{\theta_u(\kappa),M\}$, where $\theta_u(\kappa)$ is the original
constant from the uniqueness axiom for $(\cuco X, \mathfrak S)$.  Then
if $x,y \in \cuco X$ with $\dist(x,y)>\theta'_u$, then uniqueness for
$(\cuco X, \mathfrak S)$ implies there exists $U \in \mathfrak S$ with
$\dist_{\fontact U}(x,y)>M$.  Either $U \in \mathfrak T$ or
$\diam(\fontact U) = \infty$ and $\mathbf E_U$ is bounded.  We are
done in the first case.  In the second case, by construction 
the factor space $\hat U$ of $\mathbf F_U$ obtained by factoring $\frakT_{U}$ 
is quasi-isometrically embedded in $\mathcal T_{S}$ and there is a 
$1$--Lipschitz map from $\hat U$ to $\fontact U$. Thus the lower 
bound on distance in $\fontact U$ provides a lower bound on the 
distance in $\hat U$, which, in turn, provides a lower bound in 
$\mathcal T_{S}$, as desired.
\end{proof}

\begin{cor}\label{cor:betterHHG}
Every hierarchically hyperbolic group admits a 
hierarchically hyperbolic group structure with unbounded products.
\end{cor}

\begin{proof}
Recall that every hierarchically hyperbolic group has the bounded
domain dichotomy.  Accordingly, if we start with a hierarchically
hyperbolic group, $(G,\frakS)$, then
Theorem~\ref{thm:betterHHSstructure} yields a hierarchically
hyperbolic structure with unbounded products, $(G,\frakT)$, where
$\frakT$ is the structure from the proof of
Theorem~\ref{thm:betterHHSstructure} with the additional ``dummy
domains'' added as provided at the end of that proof via
Theorem~\ref{thm:almost_HHS_are_HHS}.  It remains only to show that
this is a hierarchically hyperbolic group structure.  The action of
$G$ on itself, by left multiplication, is clearly metrically proper
and cobounded, and thus it only remains to show that $\frakT$ contains
finitely many $G$--orbits.  If $U\in\s$ but $U\not\in \frakT$, then
either $\mathbf F_U$ or $\mathbf E_U$ must be bounded.  Then for each
$g\in G$, the same will be true for $\mathbf F_{gU}$ or $\mathbf
E_{gU}$, which shows that $gU\not\in\frakT$.  Thus $G\cdot U\subset
\frakS - \frakT$.  The now 
result follows from the fact that $\s$ has only finitely many 
$G$--orbits and that 
any dummy domains added fall into only finitely many orbits, as noted  
in Remark~\ref{rem:almost_HHGs_are_HHSs}. \end{proof}

\section{Characterization of contracting geodesics} \label{sec:Charact_contractions}

For this section, fix a hierarchically hyperbolic space $(\cuco X,\mathfrak S)$ with the bounded 
domain dichotomy; denote the $\nest$--maximal element $S \in \mathfrak S$.

\begin{defn}[Bounded projections]
Let $\mathcal Y \subset \cuco X$ and $D>0$.  We say that $\mathcal Y$ 
has 
\emph{$D$--bounded projections} if for every $U \in \mathfrak S -  \{S\}$,
we have $d_U(\mathcal Y) < D.$
\end{defn}

\begin{defn}[Contracting]\label{defn:contracting}
A subset $\gamma$ in a metric space $X$ is said to be 
$D$--\emph{contracting} if there exist a map $\pi_{\gamma}\co X
\rightarrow \gamma\subset X$ and constants $A, D>0$ satisfying:
\begin{enumerate}
\item For any $x \in \gamma$, we have $\diam_X(x, \pi_{\gamma}(x))< D$;

\item If $x,y \in X$ with $d_X(x,y) < 1$, then
$\diam_X(\pi_{\gamma}(x), \pi_{\gamma}(y)) < D$;

\item\label{def:contracting:item:contracting} For all $x\in X$, if we
set $R=A \cdot d(x,\gamma)$, then $\diam_{X}(\pi_{\gamma}(B_R(x)))\leq D$.
\end{enumerate}
\end{defn}

In this section, we will focus our attention to the case of 
Definition~\ref{defn:contracting} where $\gamma$ is a quasigeodesic. 
In Section~\ref{sec:stability} we will consider results about 
arbitary subsets with the contracting property.

We note that sometimes authors refer to any quasigeodesic
satisfying~(\ref{def:contracting:item:contracting}) as
\emph{contracting}. Nonetheless, for applications one also 
needs to assume the coarse idempotence and coarse Lipschitz 
properties given by (1) and (2), so for convenience we combine them 
all in one property.

A useful well-known fact is stability of
contracting quasigeodesics. Two different proofs of the following occur as 
special cases of the results  \cite[Lemma 6.1]{MasurMinsky:I} and
\cite[Theorem~6.5]{Behrstock:asymptotic}; this explicit statement 
can also be found in \cite[Section 4]{DurhamTaylor:stable}.

\begin{lem}\label{lem:contracting implies stab}
If $\gamma$ is a $D$--contracting $(K,K)$--quasigeodesic in a metric
space $X$, then $\gamma$ is $D'$--stable for some $D'$ depending only
on $D$ and $K$.
\end{lem}

The following result and argument both generalize and simplify the analogous 
result for mapping class groups in \cite{Behrstock:asymptotic}.

\begin{thm}\label{thm:char of contracting} Let $(\cuco X, \mathfrak
S)$ be a hierarchically hyperbolic space.  For any $D>0$ and $K\geq 1$
there exists a $D'>0$ depending only on $D$ and $(\cuco X, \mathfrak
S)$ such that the following holds for every 
$(K,K)$--quasigeodesic $\gamma \subset \cuco X$.  If $\gamma$ has $D$--bounded
projections, then $\gamma$ is $D'$--contracting.  Moreover, if $(\cuco
X,\s)$ has the bounded domain dichotomy, then
$\cuco X$
admits a hierarchically hyperbolic structure $(\cuco X, \mathfrak T)$ with unbounded
products where, additionally, we have that if $\gamma$ is
$D$--contracting, then $\gamma$ has $D'$--bounded projections.
\end{thm}

\begin{proof}  
    First suppose that $\gamma$ has $D$--bounded projections.  It
    follows immediately from the definition that $\gamma$ is a
    hierarchically quasiconvex subset of $\cuco X$.  Hierarchical
    quasiconvexity is the hypothesis necessary to apply
    \cite[Lemma~5.5]{BehrstockHagenSisto:asdim} (see Lemma~\ref{lem:gate}), which then yields
    existence of a coarsely Lipschitz \emph{gate map} $\mathfrak g
    \co\cuco X\to \gamma$, i.e., for each $x\in \cuco X$, the image
    $\mathfrak g(x)\in\gamma$ has the property that for all
    $U\in\mathfrak S$ the set $\pi_{U}(\mathfrak g(x))$ is a uniformly
    bounded distance from the projection of $\pi_{U}(x)$ to
    $\pi_{U}(\gamma)$.
    
    We will use $\mathfrak g$ as the map to prove $\gamma$ is 
    contracting.  Gate maps satisfy condition (1) of Definition
    \ref{defn:contracting} by definition and condition (2) since they
    are coarsely Lipschitz.  Hence it remains to prove that condition
    (3) of Lemma \ref{lem:contracting implies stab} holds.
          
    Fix a point $x\in\cuco X$ with $\dist_{\cuco
    X}(x,\gamma)\geq B_0$ and let $y\in\cuco X$ be any point with 
    $\dist_{\cuco X}(x,y)<B_{1}\dist_{\cuco X}(x,\gamma)$ for   
    constants $B_{0}$ and $B_{1}$ as determined below.
    
    Since $\mathfrak g$ is a gate map and $\gamma$ has $D$--bounded
    projections, for all $U\in \mathfrak S -\{S\}$ we have $\dist_{
    U}(\mathfrak g(x), \mathfrak g(y))<D$.  Thus, by taking a
    threshold $L$ for the distance formula
    (Theorem~\ref{thm:distance_formula}) larger than $D$, we have
   $$\dist_{\cuco X}(\mathfrak g(x), \mathfrak g(y)) \asymp_{(K,C)} 
   \dist_{S} (\mathfrak g(x), \mathfrak g(y)),$$
   for uniform constants $K,C$.  Thus it suffices
    to prove that $\dist_S(\mathfrak g(x),\mathfrak g(y))$ is bounded 
    by some uniform constant $B_2$. We also choose $L$ 
    to be larger than the constants in
    Definition~\ref{defn:HHS}.(\ref{item:dfs_transversal}).
       
    By Definition~\ref{defn:HHS}.(\ref{item:dfs_curve_complexes}), the
    maps $\pi_{U}$ are Lipschitz with a uniform constant.  Taking
    $B_{0}$ sufficiently large, it follows that there exists
    $U\in\mathfrak S$ so that $\dist_{U}(x, \mathfrak g(x))> L$.  By
    choosing $B_{1}$ to be sufficiently small, and applying the
    distance formula to the pairs $(x,y)$ and $(x,\mathfrak g(x))$,
    the fact that the projections $\pi_U$ are Lipschitz implies that
    the sum of the terms in the distance formula
    associated to $(x,\mathfrak g(x))$ is much greater than the sum of
    those associated to $(x,y)$.  
    Having chosen $B_{1}<\frac{1}{2}$, we have $\sum\dist_U(x,\mathfrak g(x))>2\sum
    \dist_U(x,y)>\sum(\dist_U(x,y)+L)$. Thus, 
    there exists 
    $W\in\mathfrak S$ for which $\dist_W(x,\mathfrak g(x))>
    \dist_W(x,y)+L$.
    
    If $W = S$, then having $\dist_S(x,\mathfrak g(x))> 
    \dist_S(x,y)+L$ (where we enlarge $L$ if necessary) would already 
    show that the $\fontact S$--geodesic between $x$ 
    and $y$ was disjoint from $\pi_{S}(\gamma)$ and then 
    hyperbolicity of $\fontact S$ would yield a uniform bound on the 
    $\dist_S(\mathfrak g(x),\mathfrak g(y))$.
    
    Otherwise, we may assume $W\neq S$.  By the triangle inequality,
    we have $\dist_W(y,\mathfrak g(x))>L$.  Further, since, as
    noted above, the $\fontact W$ projections between $\mathfrak g(x)$
    and $\mathfrak g(y)$ are uniformly bounded, by choosing $B_0$
    large enough and $B_{1}$ small enough, we also have
    $\dist_W(y,\mathfrak g(y))>L$.
    
    By the bounded geodesic image axiom
    (Definition~\ref{defn:HHS}.(\ref{item:dfs:bounded_geodesic_image})),
    any geodesic in $\fontact S$ either has bounded projection to
    $\fontact U$ or satisfies
    $\pi_{S}(\gamma)\cap\neb_E(\rho^U_S)\neq\emptyset$ for any $U \in
    \mathfrak S - \{S\}$.  For any geodesic from $\pi_S(x)$ to
    $\pi_S(\mathfrak g(x))$ (or from $\pi_S(y)$ to $\pi_S(\mathfrak
    g(y)$), the above argument implies that the first condition
    doesn't hold for $W$.  Thus, in both cases, we know that any such geodesic
    must pass uniformly close to $\rho^W_S$.  Hence the hyperbolicity
    of $\fontact S$ implies $\gamma$ is contracting, and the first
    implication holds.
           
    We prove the second implication by contradiction. 
    By Theorem~\ref{thm:betterHHSstructure}, we obtain a new structure 
    $(\cuco X, \mathfrak T)$ which has unbounded products. 
    For every $U\in\frak T - 
    \{S\}$ we have that both $\FU {U}$ and $\EU U$ are unbounded,
    hence every $U\in\frakT - \{S\}$ yields a non-trivial
    product region $\mathbf P_U=\EU {U}\times \FU {U}$ which is uniformly 
    quasi-isometrically embedded in $\cuco X$.  
    
    Suppose $\gamma$ is contracting but
    doesn't have 
    $D$--bounded projections.  Then we obtain a sequence 
    $\{U_{i}\}\in\mathfrak T - \{S\}$ with $\diam(\pi_{\fontact  
    U_{i}}(\gamma))\to\infty$. 
    Thus there is a sequence of pairs of points $x_i, y_i \in \gamma$,
    so that $d_{U_i}(x_i, y_i) \asymp K_i$, with $K_i \rightarrow \infty$.  For each $i$,
    let $q_i$ be a $R$--hierarchy path between $x_i, y_i$.  By
    \cite[Proposition 5.17]{BehrstockHagenSisto:HHS_II}, there exists
    $\nu>0$ depending only on $R$ and $(\cuco X,\mathfrak S)$, such
    that
    $$\diam_{U_i}\left(q_i \cap \mathcal{N}_{\nu}(\PU{U_i})\right) \asymp K_i.$$
    
    Since $\gamma$ is contracting, it is uniformly stable by Lemma
    \ref{lem:contracting implies stab}.  Since $\gamma$ is uniformly 
    stable and the $q_i$ are uniform
    quasigeodesics, it follows 
    that each $q_i$ is contained in a uniform neighborhood of
    $\gamma$.  Hence arbitrarily long segments of $\gamma$ are
    uniformly close to the product regions $\PU{U_i}$. This 
    contradicts the assumption that $\gamma$ is contracting and  
    completes the proof.
    \end{proof}

\section{Universal and largest acylindrical actions} \label{sec:uaa}

The goal of this section is to show that for every hierarchically
hyperbolic group $(G,\s)$ the poset $\mathcal{AH}(G)$ has
a largest element.  Recall that the action associated to such an
element is
necessarily a universal acylindrical action.  

We prove the
following stronger result which, in 
addition to providing new largest and universal acylindrical
actions for cubulated groups, gives a single
construction that recovers all previously known largest and
universal acylindrical actions of finitely presented groups that are
not relatively hyperbolic.

The following is Theorem
\ref{thmi:largestaction} of the introduction: 

\begin{thm} \label{thm:largest}
Every hierarchically hyperbolic group  
admits a largest acylindrical action. 
\end{thm}

Before giving the proof, we record the following result which gives 
a sufficient condition for an action to be largest.  
This result follows directly from the proof of \cite[Theorem~4.13]{AbbottBalasubramanyaOsin:extensions}; we give a 
sketch of the argument here.  Recall that an action $H\curvearrowright S$ is \emph{elliptic} if $H$ has bounded orbits.

\begin{prop}  
    [\cite{AbbottBalasubramanyaOsin:extensions}] \label{largest}
Let $G$ be a group, $\{H_1,\dots,H_n\}$ a finite collection of subgroups of $G$, and $F$ be a finite subset of $G$ such that $F\cup (\bigcup_{i=1}^n H_i)$ generates $G$.  Assume that:
	\begin{enumerate}[(1)]
	\item $\Gamma(G,F\cup(\bigcup_{i=1}^n H_i) )$ is hyperbolic and 
	the action of $G$ on it is acylindrical.
	\item Each $H_i$ is elliptic in every acylindrical action of $G$ on a hyperbolic space.	
	\end{enumerate}
Then $[F\cup(\bigcup_{i=1}^n H_i)]$ is the largest element in  $\mathcal{AH}(G)$.
\end{prop}

\begin{proof}
First notice that by assumption (1), $\Gamma(G,F\cup(\bigcup_{i=1}^n
H_i)$ is an element of $\mathcal{AH}(G)$.  Let $G\curvearrowright S$
be a cobounded acylindrical action of $G$ on a hyperbolic space, $S$, and
fix a basepoint $s\in S$.  Then there exists a bounded subspace
$B\subset S$ such that $S\subseteq \bigcup_{g\in G} g\cdot B$.  By
assumption (2), the orbit $H_i\cdot s$ is bounded for all
$i=1,\dots,n$.  Since $|F|<\infty$, we know $\diam(F\cdot s)< \infty$ 
and thus

$$K=\max\{\diam(B),\diam(H_1\cdot
s),\dots,\diam(H_n\cdot s), \diam(F\cdot s)\}$$ is finite. 
Let $C=\{s'\in S\mid \dist(s',s)\leq 3K\}$, and let
$$Z=\{g\in G\mid g\cdot C\cap C\neq \emptyset\}.$$ 

The standard Milnor-Schwartz
Lemma argument shows that $Z$ is an infinite generating set of $G$ and 
there exists a $G$--equivariant quasi-isometry $S\to\Gamma(G,Z)$.  It is
clear that $Z$ contains $F$, as well as $H_i$ for all $i=1,\dots,n$ 
and thus $[Z]\preceq [F\cup(\bigcup_{i=1}^n H_i]$.  The result
follows.
\end{proof}

\begin{proof}[Proof of Theorem~\ref{thm:largest}]
Let $(G,\s)$ be a hierarchically hyperbolic group with finite
generating set $F$.  By 
Corollary \ref{cor:betterHHG}, there is a hierarchically hyperbolic
group structure $(G,\frakT)$ with unbounded products.  Recall that $S$
is the $\nest$--maximal element of $\frakT$ with associated hyperbolic
space $\mathcal T_S$.  The action on $\mathcal T_S$ is acylindrical by
\cite[Theorem~K]{BehrstockHagenSisto:HHS_I}.

Moreover, the action of $G$ on $\mathcal T_S$ is cobounded, so let $B$ be a
fundamental domain for $G\curvearrowright \mathcal T_S$ and $$\mathcal
U=\{U\in \frakT\mid \pi_{S}(\FU U)\subset B \textrm{ and $U$ is
$\nest$--maximal in $\frakT - \{S\}$}\}.$$ Notice that $\mathcal
U$ will contain exactly one representative from each $G$--orbit of domains, and so must be a finite set.  Indeed, for a hierarchically hyperbolic
group, this follows from the fact that the action of $G$ on $\frakT$
is cofinite.  

Let $H_i\leq G$ be the stabilizer of $\mathbf F_{U_i}$ for each $U_i\in \mathcal U$.  By a standard Milnor-Schwartz argument (see \cite{AbbottBalasubramanyaOsin:extensions}
for details) there is a $G$--equivariant quasi-isometry between
$\Gamma(G,F\cup(\bigcup_{i=1}^n H_i))$ and $\mathcal T_S$, where $n=|\mathcal U|$.  Therefore
condition (1) of Proposition \ref{largest} is satisfied.

By definition, each $H_i$ sits inside a non-trivial direct product in
$G$, the product region $\mathbf P_{U_i}$ associated to each $U_i\in\mathcal U$.   It follows that $H_i$ must be elliptic in every acylindrical action of $G$ on a
hyperbolic space (see Remark \ref{products}), satisfying condition (2).

Therefore, by Proposition \ref{largest}, the action is largest.
\end{proof}

    \begin{rem} The proof of Theorem~\ref{thm:largest} can be extended
    to treat a number of groups which are hierarchically hyperbolic
    spaces, but not hierarchically hyperbolic groups.  For example, it
    was shown in \cite[Theorem~10.1]{BehrstockHagenSisto:HHS_II} that
    every fundamental group of a compact $3$--manifold with no Nil or Sol in
    its prime decomposition admits a hierarchically hyperbolic space
    structure, which is constructed by first putting an HHS structure on each geometric piece in the prime decomposition.  However, as explained in
    \cite[Remark~10.2]{BehrstockHagenSisto:HHS_II} it is likely that
    such fundamental groups don't all admit hierarchically hyperbolic group structures.
    Nonetheless, the proof of the above theorem works in this case by
    replacing the use of the fact that the action of $G$ on $\frakT$
    is cofinite, with the fact that for $\pi_1(M)$, the set $\mathcal U$ is
    precisely the set of $\nest$--maximal domains in the
    hierarchically hyperbolic structure on each of the Seifert-fibered
    components of the prime decomposition of $M$, and so is finite.
\end{rem}

\begin{rem} \label{rem:universal}
 There is an instructive direct proof of the universality of the above
 action, using the characterization of contracting quasigeodesics in Section
 \ref{sec:Charact_contractions}, which we now give.  We call an 
 infinite order element \emph{contracting} if its orbit is a 
 contracting quasigeodesic in the Cayley graph. Now, let $g\in G$ be
 an infinite order element and consider the geodesic $\langle
 g\rangle$ in $\Gamma(G,F)$. 

If $\langle g\rangle$ is contracting in $\Gamma(G,F)$, then by Theorem \ref{thm:char of contracting} all proper projections are bounded, and thus by the distance formula, $g$ is loxodromic for the action on $\mathcal T_S$.

If $\langle g\rangle$ is not contracting in $\Gamma(G,F)$, then there
exists some $U\in\frakT$ such that $\pi_U(\langle g\rangle)$ is
unbounded.  Thus for any increasing sequence of constants $(K_i)$ with
$K_i\to\infty$, there are sequences of pairs of points
$x_i,y_i\in\langle g\rangle$ such that $\dist(x_i,y_i)\to\infty $ as
$i\to\infty$ and $\dist_U(x_i,y_i)\geq K_i$.  For each $i$, let
$\gamma_i$ be an $R$--hierarchy path between $x_i$ and $y_i$.  By
definition, $\gamma_i$ is a uniform quasigeodesic.  Then by
\cite[Proposition 5.17]{BehrstockHagenSisto:HHS_II}, there exists
$\nu>0$ depending only on $R$ and $(\X,\frakT)$ such that
$\diam_U(\gamma_i\cap N_\nu(\mathbf P_U))\geq K_i$.  If $g$ is a
generalized loxodromic, then $\langle g\rangle$ is stable, by
\cite{Sisto:qconvex}, and so the subgeodesic $[x_i,y_i]$ stays within
a uniform bounded distance of $\gamma_i$.  Thus arbitrarily long
subgeodesics of $\langle g\rangle$ stay within a uniformly bounded
distance of a product region, $\mathbf P_U$.  This contradicts
$\langle g\rangle$ being Morse, and therefore $g$ is not a generalized
loxodromic element.
\end{rem}

This remark directly implies that the action on $\mathcal T_S$ is a universal 
acylindrical action. (The universality of the action can also be proven using the
classification of elements of $\Aut(\s)$ described in
\cite{DurhamHagenSisto:HHS_boundary}.)

Another immediate consequence of the above remark is the following,
which for hierarchically hyperbolic groups strengthens a result
obtained by combining Osin \cite[Theorem 1.4.(L4)]{Osin:acyl} and
Sisto \cite[Theorem 1]{Sisto:qconvex}, which together prove that a
generalized loxodromic element in an acylindrically hyperbolic group is
quasi-geodesically stable.

\begin{cor} \label{cor:contracting}
    Let  $(G,\s)$ be
    a hierarchically hyperbolic group. An element $g\in G$ is
    generalized loxodromic if and only if $g$ is
    contracting.
\end{cor}

The next result provides information about the partial ordering of
acylindrical actions.  Of the groups listed below, the largest and
universal acylindrical action of the class of special CAT(0) cubical groups is
new; the other cases were recently established to be largest in
\cite{AbbottBalasubramanyaOsin:extensions}.

\begin{cor} \label{cor:largestexamples}
The following groups admit acylindrical actions that are largest (and therefore universal):
\begin{enumerate}[(1)]
\item Hyperbolic groups.
\item Mapping class groups of orientable surfaces of finite type. 
\item Fundamental groups of compact three-manifolds with no Nil or Sol
in their prime decomposition.  
\item Groups that act properly and
cocompactly on a special CAT(0) cube complex, and more generally any 
cubical group which admits a factor system.  This includes right
angled-Artin groups, right-angled Coxeter groups, and many other 
examples as in \cite{HagenSusse:HHScubical}.
\end{enumerate} 
\end{cor}

\begin{proof}
With the exception of (3), by \cite{BehrstockHagenSisto:HHS_I,
BehrstockHagenSisto:HHS_II, HagenSusse:HHScubical} the above are all hierarchically
hyperbolic groups and therefore have
the bounded domain dichotomy. In case (3), where $G$ is the fundamental group of a
compact three-manifold with no Nil or Sol in its prime decomposition, then
while $G$ is not always a hierarchically hyperbolic group, it has a
hierarchically hyperbolic structure $(\X,\s)$.  
To see this, we use the fact that there is a group $G'$ which 
is quasi-isometric to $G$ and has a hierarchically hyperbolic structure 
with all of the associated hyperbolic spaces 
infinite \cite[Theorem~10.1 \& Remark~10.2]{BehrstockHagenSisto:HHS_II}; 
thus by quasi-isometric invariance of hierarchically hyperbolic 
structures \cite[Proposition~1.10]{BehrstockHagenSisto:HHS_II}, 
$G$ does as well. 
Since all of the
associated hyperbolic spaces are infinite, $(\X,\s)$ has
the bounded domain dichotomy, so the result follows.
\end{proof}

We give an explicit description of these actions for each
hierarchically hyperbolic group in the corollary, in the sense that we
describe the set $\W$ of domains which are removed from the standard
hierarchical structure of the group  and whose associated hyperbolic 
space is infinite diameter.  Recall that the space $\mathcal
T_S$ is constructed from $\cuco X$ by coning off all elements of
$\frakT$ which consists of those components of $\s$ whose associated 
product regions have both factors with infinite diameter. Coning off 
all of $\frakT$ yields a space which is is quasi-isometric to the space obtained by just 
coning off $\s - \W$.

\begin{enumerate}[(1)]

\item  Hyperbolic groups have a canonical 
simplest hierarchically hyperbolic group structure given by taking 
$\s=\{S\}$, where $\fontact S$ is the Cayley graph 
of the group with respect to a finite generating set.  For this 
structure, $\W=\emptyset$, and the action on the Cayley graph is 
clearly largest. 

\item For mapping class groups, the natural hierarchically hyperbolic
group structure is $\s$ is the set of homotopy classes of non-trivial  
non-peripheral (possibly disconnected) subsurfaces of the surface; the
maximal element $S$
is the surface itself, and the hyperbolic space $\fontact S$ is
the curve complex.  For this structure $\W=\emptyset$. (Note that to 
form $\frakT$ one must remove the 
nest-maximal collections of disjoint subsurfaces; the 
hyperbolic space associated to each of 
these, except $S$, has finite diameter). Additionally, we emphasize 
that although  
the new hyperbolic space 
$\mathcal T_{S}$ is not $\fontact S$, it is quasi-isometric to 
$\fontact S$, the action on which  is known to be universal.  Universality of this action was shown by Osin in \cite{Osin:acyl}, and follows from results of Masur-Minsky and Bowditch 
\cite{Bowditch:tight,MasurMinsky:I}.  
     
\item If $M$ is a compact $3$--manifold with no Nil or Sol in its prime decomposition and $G=\pi_1M$, then $\W$ is exactly the set of vertex groups in
the prime decomposition that are fundamental groups of hyperbolic
3--manifolds (each of which has exactly one domain in its
hierarchically hyperbolic structure). 

\item If $G$ is a group that acts properly and cocompactly on a
special CAT(0) cube complex $X$, then by \cite[Proposition~B]{BehrstockHagenSisto:HHS_I}, $X$ has
a $G$--equivariant factor system.  This factor system gives a
hierarchically hyperbolic group structure in which $\s$ is the closure under projection of the set of
hyperplanes along with a maximal element $S$, where $\fontact S$
is the contact graph as defined in \cite{Hagen:quasi_arboreal}.  In this structure, $\W$
is the set of indices whose stabilizer in $G$ contains a
power of a rank one element.

In the particular case of right-angled Artin groups, no power of a
rank one element will stabilize a hyperplane, so $\W=\emptyset$.  In
this case, the contact graph $\fontact S$ is quasi-isometric to the
extension graph defined by \cite{KimKoberda:curve_graph}.  That the
action on the extension graph is a universal acylindrical action
follows from the work of \cite{KimKoberda:curve_graph} and the
centralizer theorem for right-angled Artin groups.  This action is
also shown to be largest in
\cite{AbbottBalasubramanyaOsin:extensions}.

\medskip 

We give a concrete example of the situation in the case of a right-angled Coxeter group. 

\begin{exmp}
Let $G$ be the right-angled Coxeter group whose defining graph is a
pentagon.  Then $G=\langle 
a,b,c,d,e\mid[a,b],[b,c],[c,d],[d,e],[a,e], 
a^{2},b^{2},c^{2},d^{2},e^{2}
\rangle$, and the Cayley graph of $G$ is the tiling of the hyperbolic
plane by pentagons.  We consider the dual square complex to this
tiling.
To form the contact graph $\fontact S $, we start with the square
complex and cone off each hyperplane carrier, which is equivalent to
coning off the hyperplane stabilizers in the Cayley graph.  The result
is a quasi-tree.  Thus a fundamental domain for the hierarchically hyperbolic group structure
of $G$ is $\{U_a,U_b,U_c,U_d,U_e,S\}$ where $U_v$ is associated to
the stabilizer of the hyperplane labeled by $v$ and $S$ is associated
to the contact graph described above.

Consider the hyperplane $J_b$ that is labeled by $b$.  Then the stabilizer of $J_b$ is subgroup generated by the star of the vertex $b$, that is $\langle a,b,c\rangle$.  This subgroup contains the infinite order element $ac$.  As $G$ is a hyperbolic group, all infinite order elements are generalized loxodromic, but $ac$ is not loxodromic for the action on the contact graph since its axis lies in a hyperplane stabilizer that has been coned-off.  Thus the action on the contact graph is not universal.

Let $U_b\in \s$ be the element associated to
$\operatorname{Stab}(J_b)$.  Then $\operatorname{Stab}(J_b)=\langle
a,b,c\mid [a,b],[b,c]\rangle \simeq D_\infty\times\mathbb Z/2\mathbb Z
\simeq \FU {U_b}\times \EU {U_b}$ is a product region, and the maximal
orthogonal component $\EU {U_b}$ is bounded.  Thus $U_b\in\W$, as is
$U_v$ for each vertex $v$ of the defining graph.  The contact graph
associated to $(\FU {U_b},\s_{U_b})$ is a line, and the element $ac$
is loxodromic for the action on this space.

Note that once $\W$ has been removed from $\s$, the resulting hierarchically hyperbolic structure is $(G,\{S\})$, the canonical hierarchically hyperbolic structure for a hyperbolic group, in which $\fontact S=\Gamma(G,\{a,b,c,d,e\})$.
\end{exmp}

\end{enumerate}

\section{Characterizing stability}\label{sec:stability}

In this section, we will give several characterizations of stability
which hold in any hierarchically hyperbolic group.  In fact, we will
characterize stable embeddings of geodesic metric spaces into
hierarchically hyperbolic spaces with unbounded products.  One
consequence of this will be a description of points in the Morse boundary
of a proper geodesic hierarchically hyperbolic space with unbounded
products as the subset of the hierarchically hyperbolic boundary 
consisting of points with bounded projections.

\subsection{Stability}

While it is well-known that contracting implies stability 
\cite{Behrstock:asymptotic, DrutuMozesSapir:Divergence, MasurMinsky:I}, the converse
is not true in general.  Nonetheless, in several important classes of 
spaces the converse holds, including in  
hyperbolic spaces, CAT(0) spaces, the mapping class group, and
Teichm\"uller space \cite{Sultan:hypquasiCAT(0), Behrstock:asymptotic,
DurhamTaylor:stable,Minsky:quasiproj}.  We record the following corollary of Theorem~\ref{thm:char of contracting} which gives a 
relationship between stability and contracting 
subsets 
that holds in a fairly general context.

\begin{cor}\label{cor:char of stab 1}
Suppose that $(\cuco X, \mathfrak S)$ has unbounded products, $\cuco Y$ is a hyperbolic metric space, and  $i\colon \cuco Y \to \cuco X$ is a $(K,C)$--quasi-isometric embedding.  Then $i(\cuco Y)$ is $N$--stable if and only if $i(\cuco
Y)$ is $D$--contracting, where $N$ and $D$ determine each other.
\end{cor}

\begin{proof}
    First assume that $i(\cuco Y)$ is $D$--contracting.  Since $i\colon \cuco Y \to \cuco X$ is a $(K,C)$--quasi-isometric embedding, to show that $i(\cuco Y)$ is $N$--stable for some gauge  $N=N(D)$, we need only show that the (quasigeodesic) image $i(\gamma)$ of every geodesic $\gamma$ in $\cuco Y$ is $N(K,C)$--stable.  Since $i(\cuco Y)$ is $D$--contracting and $i(\cuco Y)$ is hyperbolic, $i(\gamma)$ is $D'$--contracting for some $D'$ depending only on $D,K,C,$ and the hyperbolicity constant of $\cuco Y$.  Lemma~\ref{lem:contracting implies stab} shows that $i(\gamma)$ is therefore $N$--stable, with $N$ depending only on $D$, as desired.  (Note that the assumption on unbounded products is not necessary for this implication.)
    
    For the other direction, the 
    fact that $\cuco X$ has unbounded products implies that $i(\cuco Y)$ 
    has bounded projections, since otherwise one could find large 
    segments of quasigeodesics contained inside product regions with unbounded factors, 
    contradicting stability. The result now follows from 
    Theorem \ref{thm:char of contracting}.
\end{proof}

The following provides a general characterization of stability in HHSs, a special case of which is Theorem \ref{thmi:stab}.

\begin{cor}\label{cor:char of stab 2}
Let $i\co\cuco Y \rightarrow \cuco X$ be a quasi-isometric embedding from a metric space into
a hierarchically hyperbolic space $(\cuco X, \mathfrak S)$ with unbounded products.  The following
are equivalent:

\begin{enumerate}
\item\label{cor:stab2:enum1} $i$ is a stable embedding;
\item\label{cor:stab2:enum2} $i(\cuco Y)$ has uniformly bounded projections;
\item\label{cor:stab2:enum3} $\pi_S \circ i\co \cuco Y \rightarrow \fontact S$ is a quasi-isometric embedding.
\end{enumerate}
\end{cor}

\begin{proof}
That item (2) implies (3) 
follows from the
distance formula and the assumption that  $i$ is a quasi-isometric
embedding.

The hypothesis of item (1) implies that $\cuco Y$ is hyperbolic.  Moreover, since (2) implies (3), the hypothesis of (2) also implies that $\cuco Y$ is hyperbolic.  Thus items (1) and (2) are equivalent via Corollary~\ref{cor:char of stab 1} and Theorem~\ref{thm:char of contracting}.  

We now prove that (3) implies (2). Suppose for a contradiction that for any integer $N$ 
there exists $U\in\frakS - \{S\}$ and $x,y\in i(\cuco Y)$ 
satisfying $d_{U}(x,y)>N$. 
Now, we consider a hierarchy path $\gamma$ between $x$ and $y$.
Applying 
the bound geodesic image axiom (Definition~\ref{defn:HHS}.(7)) 
to the associated $\fontact S$--geodesic between $\pi_S \circ 
i(x)$ and $\pi_S \circ i(y)$ it follows that this $\fontact S$--geodesic 
has non-trivial intersection with the 
radius $E$ ball about the point $\rho^{U}_{S}$. Indeed, this yields 
that there exist points $x',y'$ on the geodesic which are both 
distance at most $E$ from $\rho^{U}_{S}$; by 
\cite[Lemma~5.17]{BehrstockHagenSisto:HHS_II} we can assume that $x$ 
and $y$ were chosen so that $x'$ and $y'$ also satisfy $d_{
S}(x,x')<E$ and $d_{S}(y,y')<E$. Thus, we have that 
$d_{S}(x,y)<4E$. The hypothesis in (3) implies that there is a uniform bound on 
$d_{\cuco Y}(x,y)$. The distance formula then implies a uniform bound 
on $d_{W}(x,y)$ for any $W\in\frakS$, contradicting the fact that we 
chose $d_{U}(x,y)$ to be large. 
\end{proof}

\subsection{The Morse boundary}\label{subsec:cobounded boundary}

In the rest of this section, we turn to studying the Morse boundary
and use this to give a bound on the stable asymptotic dimension of a 
hierarchically hyperbolic space.  We begin by describing two notions of boundary.

In \cite{DurhamHagenSisto:HHS_boundary}, Durham, Hagen, and
Sisto introduced a boundary for any hierarchically hyperbolic space.
We collect the relevant properties we need in the following
theorem:

\begin{thm}[Theorem 3.4 and Proposition 5.8 in \cite{DurhamHagenSisto:HHS_boundary}]\label{thm:boundary}
If $(\cuco X, \mathfrak S)$ is a proper hierarchically hyperbolic space, then there exists a topological space $\partial \cuco X$ such that $\partial \cuco X \cup \cuco X = \overline{\cuco X}$ compactifies $\cuco X$, and the action of $\Aut(\cuco X, \mathfrak S)$ on $\cuco X$ extends continuously to an action on $\overline{\cuco X}$.

Moreover, if $\cuco Y$ is a hierarchically quasiconvex subspace of
$\cuco X$, then, with respect to the induced hierarchically hyperbolic
structure on $\cuco Y$, the limit set of $\Lambda \cuco Y$ of $\cuco
Y$ in $\partial \cuco X$ is homeomorphic to $\partial \cuco Y$ and the
inclusion map $i\co \cuco Y \rightarrow \cuco X$ extends continuously an
embedding $\partial i\co \partial \cuco Y \rightarrow \partial \cuco X$.
\end{thm}

Building on ideas in \cite{charneysultan}, Cordes introduced the 
\emph{Morse boundary} of a proper geodesic
metric space
\cite{cordes2015morse}, which was then refined further 
by Cordes--Hume in
\cite{cordes2016stability}. The \emph{Morse boundary} is a stratified boundary which encodes the asymptotic
classes of Morse geodesic rays based at a common point.
Importantly, it is a quasi-isometry invariant and generalizes the
Gromov boundary of a hyperbolic space \cite{cordes2015morse}.

We briefly discuss the construction of the Morse boundary and refer the reader to \cite{cordes2015morse, cordes2016stability} for details.

Consider a a proper geodesic metric space $X$ with a basepoint $e \in 
X$. Given a
stability gauge $N\co\mathbb{R}^2_{\geq 0} \rightarrow
\mathbb{R}_{\geq 0}$, define a subset
$X^{(N)}_e \subset X$ to be the collection of points $y \in X$ such
that $e$ and $y$ can be connected by an $N$--stable geodesic in $X$.
Each such $X^{(N)}_e$ is $\delta_N$--hyperbolic for some $\delta_N>0$
depending on $N$ and $X$ \cite[Proposition 3.2]{cordes2016stability};
here, we use the Gromov product definition of hyperbolicity, as
$X^{(N)}_e$ need not be connected.  Moreover, any stable subset of $X$
embeds in $X^{(N)}_e$ for some $N$ \cite[Theorem
A.V]{cordes2016stability}.

The set of stability gauges admits a partial order: $N_1 < N_2$ if and
only if $N_1(K,C) < N_2(K,C)$ for all constants $K,C$.
In particular, if $N_1 < N_2$, then $X^{(N_1)}_e \subset X^{(N_2)}_e$.

Since each $X^{(N)}_e$ is Gromov hyperbolic, each admits a Gromov
boundary $\partial X^{(N)}_e$.  Take the direct limit with
respect to this partial order to obtain a topological space
$\partial_s X$ called the \emph{Morse boundary} of $X$.

We fix $(\cuco X, \mathfrak S)$, a 
hierarchically hyperbolic structure with unbounded products.

\begin{defn}
We say $\lambda \in\partial \X$ has \emph{bounded projections} if for 
any $e \in \cuco X$, there exists $D>0$ such that any $R$--hierarchy
path $[e,\lambda]$ has $D$--bounded projections.
Let $\partial_{c} \cuco X$ denote the set of
points $\lambda \in \partial \cuco X$ with bounded projections. 
\end{defn}

The boundary $\partial \cuco X$ contains $\partial 
\fontact U$ for each $U\in\frakS$, by construction. The next lemma shows that the 
boundary points with bounded projections are contained in 
$\partial\fontact S$, as a subset of $\partial \cuco X$, where $S$ is 
the $\nest$--maximal element.
In general, the set of  boundary points with bounded projections may be a very small
subset of $\partial\fontact S$.  For instance, in the
boundary of the Teichm\"{u}ller metric, these points are a proper
subset of the uniquely ergodic ending laminations and have measure
zero with respect to any hitting measure of a random walk on the
mapping class group.

\begin{lem} \label{lem:cobdd inside CS}
The inclusion $\partial_{c} \cuco X \subset \partial \fontact S$ holds for any $(\cuco X, \mathfrak S)$ with 
unbounded products where $S$ is the $\nest$--maximal element of 
$\frak S$.  Moreover, if $\cuco X$ is also proper,
then for any $D>0$ there exists $D'>0$ depending only on $D$ and
$(\cuco X, \mathfrak S)$ such that if $(x_n) \subset \cuco X$ is a
sequence with $x_n \rightarrow \lambda \in \partial \cuco X$ such that
$[e,x_n]$ has $D$--bounded projections for some $e \in \cuco X$ and
each $n$, then $[e, \lambda]$ has $D'$--bounded projections.
\end{lem}

\begin{proof}
Let $\lambda \in \partial_{c} \cuco X$.  If $[e,\lambda]$ is an
$R$--hierarchy path, then $[e,\lambda]$ has an
infinite diameter projection to some $\fontact U$, see, e.g., \cite[Lemma
3.3]{DurhamHagenSisto:HHS_boundary}.  As $\lambda$ has
bounded projections, we must have $U = S$.  Since $\pi_S([e,\lambda])
\subset \fontact S$ is a quasigeodesic ray, the first statement 
follows.

Now suppose that $\cuco X$ is also proper.  For each $n$, let $\gamma_n=[e,x_n]$ be
any $R$--hierarchy path between $e$ and $x_n$ in $\cuco X$.  The Arzela-Ascoli theorem
implies that after passing to a subsequence, 
$\gamma_n$ converges uniformly on compact sets to some $R'$--hierarchy path 
$\gamma$  with $R'$ depending only on $R$ and $(\cuco X, \mathfrak S)$.  Hence $\gamma$ has $D'$--bounded projections
for some $D'$ depending only on $D$ and $(\cuco X, \mathfrak S)$.  Moreover, since $x_n
\rightarrow \lambda$ in $\fontact S$, it follows that $\pi_S(\gamma)$
is asymptotic to $\lambda$ in $\fontact S$.

If $[e,\lambda]$ is any other $R'$--hierarchy path, it follows 
from uniform hyperbolicity of the $\fontact U$ and the definition of
hierarchy paths that 
$d^{Haus}_U(\gamma, [e,\lambda])$ is uniformly bounded for all $U \in \mathfrak S$.  Since $\gamma$ has $D'$--bounded projections, the distance
formula implies that $[e,\lambda]$ has $D''$--bounded projections for some $D''$
depending only on $D $ and $(\cuco X,\mathfrak S)$, as required.
\end{proof}

\subsection{Bounds on stable asymptotic dimension}

The asymptotic dimension of a metric space is a coarse notion of
topological dimension which is invariant under quasi-isometry.
Introduced by Cordes--Hume \cite{cordes2016stability}, the stable
asymptotic dimension of a metric space $X$ is the maximal asymptotic
dimension a stable subspace of $X$.

The stable asymptotic dimension of a metric space $X$ is always
bounded above by its asymptotic dimension.  Behrstock, Hagen, and Sisto 
\cite{BehrstockHagenSisto:asdim} proved that all proper 
hierarchically hyperbolic spaces have
finite asymptotic dimension (and thus 
have finite stable asymptotic dimension, as well).  The bounds on
asymptotic dimension obtained in \cite{BehrstockHagenSisto:asdim} are 
functions of the asymptotic dimension of the top level curve graph. 

In the following theorem, we prove that a hierarchically hyperbolic space $(\cuco X, \mathfrak
S)$ has finite stable asymptotic dimension under the assumption
that $\mathrm{asdim} (\fontact S)< \infty$, where $\fontact S$ is the
hyperbolic space associated to the $\nest$--maximal domain $S$ in
$\mathfrak S$.
 
Recall that asymptotic dimension is monotonic under taking subsets.  Thus, if $\cuco X$ is assumed to be proper, so that $\mathrm{asdim}(\fontact S)<\infty$, then $\cuco X$ (and therefore its stable subsets) have finite asymptotic dimension by \cite{BehrstockHagenSisto:asdim}.  Here, using some geometry of 
stable subsets we obtain a sharper bound on
$\mathrm{asdim}_s (\cuco X)$ than $\mathrm{asdim} (\cuco X)$.


\begin{thm} \label{thm:stab asdim}
Let $(\cuco X, \mathfrak S)$ be a hierarchically hyperbolic space with unbounded products such 
that $\fontact S$ has finite asymptotic dimension, where $S$ is the 
$\nest$--maximal element of $\mathfrak S$.  Then 
$\mathrm{asdim}_s(\cuco X) \leq \mathrm{asdim}(\fontact S)$. 
Moreover, if $\cuco X$ is also proper and geodesic, then there exists a continuous bijection $\widehat{i}\co\partial_s \cuco X \rightarrow \partial_c \cuco X$.
\end{thm}

\begin{proof}
By \cite[Lemma 3.6]{cordes2016stability}, for any stability gauge $N$ there 
exists $N'$ such that $\cuco X^{(N)}_e$ is $N'$--stable.  Hence, there 
exists $D'>0$ depending only on $N'$ and $(\cuco X, \mathfrak S)$ 
such that $\cuco X^{(N)}_e$ has $D'$--bounded projections.  By 
Corollary~\ref{cor:char of stab 2}, it follows that the projection
$\pi_S\co\cuco X^{(N)}_e \rightarrow \fontact S$ is a quasi-isometric
embedding with constants depending only on $D'$ and $(\cuco X,
\mathfrak S)$.  Since every stable subset of $\cuco X$ embeds into
some $\cuco X^{(N)}_e$ \cite[Theorem A.V]{cordes2016stability}, the
first conclusion then follows from the definition of stable asymptotic
dimension.

Now suppose that $\cuco X$ is proper.

Since each $\cuco X^{(N)}_e$ is stable in $\cuco X$, these sets have  
bounded projections by 
Corollary~\ref{cor:char of stab 2}; from this it follows that 
$\cuco X^{(N)}_e$ is
hierarchically quasiconvex for each $N$.  Hence by \cite[Proposition~5.8]{DurhamHagenSisto:HHS_boundary}, the canonical embedding
$i^{(N)}\co \cuco X^{(N)}_e \hookrightarrow \cuco X$ extends to an
embedding $\widehat{i}^{(N)}\co \partial \cuco X^{(N)}_e \hookrightarrow
\partial \cuco X$.

By Corollary \ref{cor:char of stab 2} and Lemma \ref{lem:cobdd inside
CS}, we have $\widehat{i}^{(N)}\left(\partial \cuco X^{(N)}_e\right)
\subset \partial_c \cuco X \subset \partial \fontact S$.  Let
$\widehat{i}: \partial_s \cuco X \rightarrow \partial_c \cuco X$ be
the direct limit of the $\widehat{i}^{(N)}$.  Since it is injective on
each stratum, $\widehat{i}$ is injective.

To prove surjectivity, let $\lambda \in \partial_c \cuco X$.  Let $e \in
\cuco X$ and fix a hierarchy path $[e,\lambda]$.  Since $\lambda\in\partial_c\X$,
$[e, \lambda]$ has $D$--bounded projections for some $D>0$.  Let $x_n \in
[e,\lambda]$ be such that $x_n \rightarrow \lambda$ in
$\overline{\cuco X}$.  If $[e,x_n]$ is a sequence of geodesics between
$e$ and $x_n$, then, by properness, the Arzela--Ascoli theorem, and
passing to a subsequence if necessary, there exists a geodesic ray
$\gamma\co[0,\infty) \rightarrow \cuco X$ with $\gamma(0)=e$ such that
$[e,x_n]$ converges on compact sets to $\gamma$.  Since each $[e,x_n]$
has $D$--bounded projections, it follows that $\gamma$ has $D'$--bounded projections for
some $D'$ depending only on $D$ and $(\cuco X, \mathfrak S)$.
Moreover, by hyperbolicity of $\fontact S$ and the construction of 
$\gamma$ we have that  
$d^{Haus}_{\fontact S}(\pi_S(\gamma), [e,\lambda])$ is uniformly 
bounded and thus, by the distance formula, so is 
$d^{Haus}_{\cuco X}(\gamma, [e,\lambda])$.  Since $[(x_n)] = [\gamma]$ by construction, it
follows that $\widehat{i}(\gamma) = \lambda$, as required.

Continuity of $\widehat{i}^{(N)}$ for each $N$ follows from
\cite[Proposition 5.8]{DurhamHagenSisto:HHS_boundary}, as above. 
This and the definition of the direct limit topology implies
continuity of $\widehat{i}$.
\end{proof}

The following corollary is immediate:

\begin{cor} \label{cor:HHGstableasdim}
If $G$ is a hierarchically hyperbolic group, then $G$ has finite stable asymptotic dimension.
\end{cor}

\subsection{Random subgroups}\label{subsec:random}


Let $G$ be any countable group and $\mu$ a probability measure on $G$
whose support generates a non-elementary semigroup.  A \emph{$k$--generated 
random subgroup of $G$}, denoted $\Gamma(n)$ is defined to be the subgroup 
$\langle w^1_n, w^2_n, \dots, w^k_n\rangle \subset G$
generated by the $n^{th}$ step of $k$ independent random walks on $G$,
where $k \in \mathbb{N}$. For other recent results on the geometry of 
random subgroups of acylindrically hyperbolic groups, see 
\cite{maher2017random}.

Following Taylor-Tiozzo \cite{TaylorTiozzo:randomqi}, we say a
$k$--generated random subgroup $\Gamma(n)$ of $G$ has a property $P$ if
$$\mathbb{P}[\Gamma(n) \indent \mathrm{ has } \indent P] \rightarrow 1
\indent \mathrm{ as } \indent n\rightarrow \infty.$$

\begin{thm}\label{thm:random stable}
Let $(\cuco X,\frakS)$ be an HHS for which   
the 
$\nest$--maximal element, $S$, has $\fontact S$ infinite diameter,  
and consider 
$G<\Aut(\cuco X, \mathfrak S)$ which acts properly and cocompactly 
on~$\cuco X$ via the orbit map. Then any $k$--generated random subgroup of $G$ 
stably embeds in $\cuco X$.
\end{thm}

\begin{proof}
    
By \cite[Theorem~K]{BehrstockHagenSisto:HHS_I}, $G$ acts 
acylindrically on $\fontact S$.   
Let $\Gamma(n)$ be generated by $k$ independent random  walks as
above.  Now, \cite[Theorem 1.2]{TaylorTiozzo:randomqi} implies that
$\Gamma(n)$ a.a.s.\ quasi-isometrically embeds into $\fontact S$, and hence $\Gamma(n)$ is hyperbolic.  Moreover, the distance formula implies that $\Gamma(n)$  is undistorted in G and any orbit of $\Gamma(n)$ in $\cuco X$ has bounded projections by the distance formula.
By Theorem~\ref{thm:char of contracting}, 
having bounded projections implies contracting; thus any orbit of $\Gamma(n)$ in
$\cuco X$ is a.a.s.\ contracting, which gives the conclusion by Corollary \ref{cor:char of stab 1}.  (Note that the directions of  Theorem~\ref{thm:char of contracting} and Corollary~\ref{cor:char of stab 1} used here do not require that $(\cuco X,\mathfrak S)$ has unbounded products.)
\end{proof}

In particular, one consequence is a new proof of the 
following result of Maher--Sisto. This result follows from the above, 
together with 
Rank Rigidity for HHG (\cite[Theorem~9.14]
{DurhamHagenSisto:HHS_boundary}) which implies that a hierarchically 
hyperbolic group which is not a direct product of two infinite groups 
has $\fontact S$ infinite diameter.

\begin{cor}[Maher--Sisto; \cite{maher2017random}]\label{cor:random hhg} 
If $G$ is a hierarchically hyperbolic group which is not the direct product of two infinite 
groups, then any $k$--generated random subgroup of $G$ is stable.
\end{cor}

\section{Clean containers}

The clean container property is a condition related to the 
orthogonality axiom. In Proposition~\ref{prop:cleancontainers} this
property is shown to hold for many groups, though it does not hold for all
groups.  Unlike earlier versions of this paper, this condition is no 
longer needed to prove the main theorems of the earlier sections. 
However, we keep the content in this paper since this property has found 
independent interest and is used elsewhere.

\begin{defn}[Clean containers]\label{defn:clean containers}
In a hierarchically hyperbolic space $(\X,\s)$ for each $T\in\s$ and
each $U\in\s_T$ with $\{V\in\s_T\mid V\perp U\}\neq\emptyset$ the 
orthogonality axiom provides a container. If, for each $U$, such a 
container can be chosen to be orthogonal to $U$, then we say that 
$(\X,\s)$ has \emph{clean containers}.
\end{defn}

We first describe some interesting examples with clean containers. 
Then we show that this property is preserved under some 
combination theorems for hierarchically hyperbolic
spaces.  We refer the reader to \cite[Sections~8 \&
9]{BehrstockHagenSisto:HHS_II} and
\cite[Section~6]{BehrstockHagenSisto:asdim} for details on the 
structure in the new spaces.

\begin{prop}  \label{prop:cleancontainers}
The following spaces admit hierarchically hyperbolic structures with clean containers.
\begin{enumerate}
\item Hyperbolic groups
\item Mapping class groups of orientable surfaces of finite type
\item Special cubical groups, and more generally, any cubical group 
which admits a factor system.
\item $\pi_1(M)$, for $M$ a compact $3$--manifold with no Nil or Sol in its prime decomposition.
\end{enumerate}
\end{prop}

\begin{proof}
    Hierarchically hyperbolic structures for these spaces were 
    constructed in \cite{BehrstockHagenSisto:HHS_I} and \cite{BehrstockHagenSisto:HHS_II}.
\begin{enumerate}
\item The statement is immediate for hyperbolic groups, as they all
admit hierarchically hyperbolic structure with no orthogonality, and
thus the container axiom is vacuous.  

\item For mapping class groups,
in the standard structure, a container for domains orthogonal to a
given subsurface $U$ is the complementary
subsurface, which is orthogonal to $U$.

\item The statement follows immediately from 
\cite[Proposition~B]{BehrstockHagenSisto:HHS_I} and \cite[Corollary~3.4]{HagenSusse:HHScubical}.

\item Given a geometric 3--manifold $M$ of the above form, $\pi_1(M)$
is quasi-isometric to a (possibly degenerate) product of hyperbolic
spaces, and so has clean containers by Proposition \ref{cleanproducts}.
Given an irreducible non-geometric graph manifold $M$, the
hierarchically hyperbolic structure comes from considering $\pi_1(M)$
as a tree of hierarchically hyperbolic spaces with clean containers
and hence has clean containers by Proposition~\ref{cleantrees}.  Finally,
the general case of a non-geometric $3$--manifold $M$ follows
immediately from Proposition \ref{cleanrelhyp} and the fact that $\pi_1(M)$
is hyperbolic relative to its maximal graph manifold subgroups.
\end{enumerate}
\end{proof}

\begin{prop} 
\label{cleanproducts}
The product of two hierarchically hyperbolic spaces which both have clean containers has clean containers.
\end{prop}

\begin{proof}
Let $(\X_{0},\s_0)$ and $(\X_{1},\s_1)$ be hierarchically hyperbolic
spaces with clean containers.  In the hierarchically hyperbolic
structure $(\X_0\times\X_1,\s)$ given by \cite[Theorem~8.27]{BehrstockHagenSisto:HHS_II} there are two types of containers,
those that come from one of the original structures and those that do
not.  Containers of the first type are clean, as both original
structures have clean containers.

The second type of domain consists of new domains obtained as 
follows. Given a domain $U\in\s_i$, a new domain $V_U$ is defined 
with the property that it contains under nesting any domain in $\s_i$ 
which is orthogonal to $U$ and also any domain in $\s_{i+1}$. Thus, 
by construction $V_U$ is a container for everything orthogonal to $U$.  As $V_U\perp U$, the result follows. 
\end{proof}

\begin{prop} \label{cleanrelhyp}
If $G$ is hyperbolic relative to a collection of hierarchically
hyperbolic spaces which all have clean containers, then $G$ is a
hierarchically hyperbolic space with clean containers.
\end{prop}

\begin{proof}
That $G$ is a hierarchically hyperbolic space follows from
\cite[Theorem 9.1]{BehrstockHagenSisto:HHS_II}.  In the hierarchically
hyperbolic structure on $G$, no new orthogonality relations are
introduced, and thus all containers are containers in the
hierarchically hyperbolic structure of one of the peripheral
subgroups.  As each of these structures have clean containers, it
follows that $G$ does, as well.
\end{proof}

The following example relies on the 
combination theorem \cite[Theorem~8.6]{BehrstockHagenSisto:HHS_II}. We provide this as another example 
of hierarchically hyperbolic spaces with clean containers, but since 
we don't rely on this elsewhere in the paper, we refer to that 
reference for the relevant definitions. Nonetheless, we include a full proof 
for the expert, since it is short. (We note that after this paper was 
circulated, Berlai and Robbio proved a combination theorem under 
weaker conditions than \cite[Theorem~8.6]{BehrstockHagenSisto:HHS_II} 
and, in the process, also proved that if all the vertex spaces have 
clean containers, then so does the combined space, see \cite[Theorem~A]{BerlaiRobbio:combinationHHG}.)

\begin{prop} \label{cleantrees} Let $\mathcal T$ be a tree of hierarchically hyperbolic spaces satisfying the hypotheses of \cite[Theorem~8.6]{BehrstockHagenSisto:HHS_II}, so that $X(\mathcal T)$ is hierarchically hyperbolic.  If for each $v\in T$, the hierarchically hyperbolic space $(\mathcal X_v,\s_v)$ has clean containers, then so does $X(\mathcal T)$.
\end{prop}

\begin{proof}
This follows immediately from the proof of \cite[Theorem~8.6]{BehrstockHagenSisto:HHS_II} and the fact that
edge-hieromorphisms are full and preserve orthogonality. In the 
notation from that result, we note that, if
$\s_v$ has clean containers for each $v\in T$, then the domain
$A_v\in\s_v$ described in the proof also has the property that
$A_v\perp U_v$.  Therefore, as edge-hieromorphisms are full and
preserve orthogonality, $[A_v]\perp [U]$.
\end{proof}

The following uses the notion of \emph{hierarchically hyperbolically embedded 
subgroups} introduced in 
\cite{BehrstockHagenSisto:asdim}; see also \cite{DGO} for the related notion of hyperbolically embedded subgroups.

\begin{prop}
Let $(G,\s)$ be a hierarchically hyperbolic group with clean
containers, and let $H$ be a hierarchically hyperbolically embedded subgroup of $(G,\s)$.  Then there exists
a finite set $F\subset H-\{1\}$ such that for all $N\triangleleft H$
with $F\cap N=\emptyset$ and $H/N$ is hyperbolic,  the group
$G/\hat N$, obtained by quotienting by the normal closure, is a hierarchically hyperbolic group with clean containers.
\end{prop}

\begin{proof}
Recall that in the hierarchically hyperbolic structure $(G/\hat
N,\s_N)$ obtained in 
\cite[Theorem~6.2]{BehrstockHagenSisto:asdim} (and in the 
notation used there), two domains $\bf U,\bf V\in\s_N$ satisfy $\bf U\nest\bf V$
(respectively $\bf U\perp\bf V$) if there exists a linked pair
$\{U,V\}$ with $U\in\bf U$ and $V\in\bf V$ such that $U\nest V$
(respectively $U\perp V$).  Let $\bf T\in\s_N$ and $\bf U\in(\s_N)_T$
with $\mathcal V=\{\bf V\in\s_T\mid \bf V\perp \bf U\}\neq \emptyset$.
To prove the container axiom, we consider domains $T,U,V\in\s$ such
that $T\in\bf T$, $U\in\bf U$ and $V\in\bf V$ for all $\bf
V\in\mathcal V$, and such that any pair is a linked pair.  Then the
orthogonality axiom for $(G,\s)$ provides a domain $W$ such that
$W\sqsupseteq V$ and $W\nest T$.  As $(G,\s)$ has clean containers, we
also have that $W\perp U$.  This implies that $\rho^U_S$ and
$\rho^W_S$ are coarsely equal by \cite[Lemma 1.5]
{DurhamHagenSisto:HHS_boundary}, and so $\{U,W\}$ is a linked
pair.  Therefore, $\bf W\perp\bf U$.
\end{proof}

\input{Appendix_Almost_HHS_are_HHS_final.tex}

\bibliographystyle{alpha}
\bibliography{hhs_2015}
\end{document}

%% file: Appendix_Almost_HHS_are_HHS_final.tex

\theoremstyle{plain}
\newtheorem{thmappend}{Theorem}
\newtheorem{lemappend}[thmappend]{Lemma}
\newtheorem{claimappend}[thmappend]{Claim}
\newtheorem{corappend}[thmappend]{Corollary}

\theoremstyle{definition}
\newtheorem{defnappend}[thmappend]{Definition}
\newtheorem{remappend}[thmappend]{Remark}

\renewcommand*{\thethmappend}{A.\arabic{thmappend}}

\appendix
\section{Almost HHSs are HHSs. \\ {\small By Daniel Berlyne and Jacob Russell}}

The main result in this appendix is that every almost HHS structure can be promoted to an HHS
structure. Recall that, as introduced in Section~\ref{subsec:almost 
HHS}, an almost HHS is a space which satisfies all the axioms of 
an HHS except for the orthogonality axiom, which is instead replaced 
by a weaker axiom without a container requirement. In 
Theorem~\ref{thm:almost_HHS_are_HHS}, we show that  an almost HHS 
structure can be made into an actual HHS structure by adding 
appropriately chosen ``dummy domains'' to serve as the containers.
This result provides a useful method for producing an HHS structure 
while only needing to verify the weaker axioms of an almost HHS. This 
method is used in the main text in the proof of Theorem \ref{thm:betterHHSstructure}, where it is shown that every hierarchically hyperbolic space with the bounded domain
dichotomy admits an HHS structure with unbounded products.

\begin{thmappend}\label{thm:almost_HHS_are_HHS}
	Let $(\mathcal{X},\mathfrak{S})$ be an almost HHS. There exists 
	an HHS structure $\mathfrak{R}$ for $\mathcal{X}$ so that $\mathfrak{S}\subseteq \mathfrak{R}$, and if $W \in \mathfrak{R} - \mathfrak{S}$ then the associated hyperbolic space for $W$ is a single point.
\end{thmappend}

To prove Theorem \ref{thm:almost_HHS_are_HHS}, we will need to collect  three additional tools about almost HHSs. Each of these tools was proved in the setting of hierarchically hyperbolic spaces, but they continue to hold in the almost HHS setting. Indeed, the only use of the containers in their proofs is \cite[Lemma 2.1]{BehrstockHagenSisto:HHS_II}, which proves that the cardinality of any collection of pairwise orthogonal domains is uniformly bounded by the complexity of the HHS.

The first  tool says the relative projections of orthogonal domains coarsely coincide. Note, $\rho_Q^W$ and $\rho_Q^V$ are both defined when $W \trans Q$ or $W \propnest Q$ and $V \trans Q$ or $V \propnest Q$.

\begin{lemappend}[{\cite[Lemma 1.5]{DurhamHagenSisto:HHS_boundary}}]\label{lem:orthongality_and_rel_projections}
	Let $(\mathcal{X},\mathfrak{S})$ be an almost HHS. If $W,V\in \mathfrak{S}$  with $W \perp V$, and $Q \in \mathfrak{S}$ with $\rho_Q^W$ and $\rho_Q^V$ both defined, then $\dist_Q(\rho_Q^W,\rho_Q^V) \leq 2\kappa_0$ where $\kappa_0$ is the constant from the consistency axiom of $\mathfrak{S}$.
\end{lemappend}

The second tool we will need is the realization theorem for almost
HHSs.  The realization theorem characterizes which tuples in the
product $\prod_{V \in \mathfrak{S}} \fontact{V}$ are coarsely the
image of a point in $\mathcal{X}$.  Essentially, it says if a tuple
$(b_V) \in \prod_{V \in \mathfrak{S}} \fontact{V}$ satisfies the
consistency inequalities of an almost HHS (see 
Definition~\ref{defn:consistent_tuple}), then there exists  a point $x
\in \mathcal{X}$ such that $\pi_V(x)$ is uniformly close to $b_V$ for
each $V\in \mathfrak{S}$.



\begin{thmappend}[The realization of consistent tuples, {\cite[Theorem 3.1]{BehrstockHagenSisto:HHS_II}}]\label{thm:realization}
	Let $(\mathcal{X},\mathfrak{S})$ be an almost HHS. There exists a function $\tau\colon [0,\infty) \to [0,\infty)$ so that if $(b_V)_{V \in \mathfrak{S}}$ is a $\kappa$--consistent tuple, then there exists $x \in \mathcal{X}$ so that $\dist_V(x,b_V) \leq \tau(\kappa)$ for all $V \in \mathfrak{S}$.
\end{thmappend}

The last result we need is that the relative projections of an almost HHS also satisfy the inequalities in the consistency axiom.

\begin{lemappend}[$\rho$--consistency, {\cite[Proposition 1.8]{BehrstockHagenSisto:HHS_II}}]\label{lem:rho_consistency}
	Let $\mathfrak{S}$ be an almost HHS structure for $\mathcal{X}$ and $V,W,Q \in \mathfrak{S}$. Suppose $W \trans Q$ or $W \propnest Q$ and $W \trans V$ or $W \propnest V$. Then we have the following, where $\kappa_0$ is the constant from the consistency axiom of $(\mathcal{X}, \mathfrak{S})$.
	\begin{enumerate}
		\item If $Q \trans V$, then $\min\{\dist_Q(\rho_Q^W,\rho_Q^V), \dist_V(\rho_V^Q,\rho_V^W) \} \leq 2\kappa_0$.
		\item If $Q \nest V$, then $\min\{\dist_V(\rho_V^Q,\rho_V^W), \diam(\rho_Q^W\cup\rho_Q^V(\rho_V^W)) \} \leq 2\kappa_0$.
	\end{enumerate}
\end{lemappend}

We are now ready to prove  that every almost HHS is an HHS (Theorem \ref{thm:almost_HHS_are_HHS}). If $(\mathcal{X},\mathfrak{S})$ is an almost HHS, then the only HHS axiom that is not satisfied is the container requirement of the orthogonality axiom. The most obvious way to address this is to add an extra element to $\mathfrak{S}$ every time we need a container. That is, if $V,W \in \mathfrak{S}$ with $V \nest W$ and there exists some $Q \nest W$ with $Q \perp V$, then we add a domain $D_W^V$ to serve as the container for $V$ in $W$, i.e., every $Q$ nested into $W$ and orthogonal to $V$ will be nested into $D_W^V$. However, this approach is perilous as once a domain $Q$ is nested into $D_W^V$, we may now need a container for $Q$ in $D_W^V$! To avoid this, we add domains $D_W^\mathcal{V}$ where $\mathcal{V}$ is a pairwise orthogonal set of domains nested into $W$; that is, $D_W^\mathcal{V}$ contains all domains $Q$ that are nested into $W$ and orthogonal to all $V \in \mathcal{V}$. This allows for all the needed containers to be added at once, avoiding an iterative process.

\begin{proof}[Proof of Theorem \ref{thm:almost_HHS_are_HHS}]
	Let $(\mathcal{X},\mathfrak{S})$ be an almost HHS and let $E \geq 0$ be the maximum of all the constants in $\mathfrak{S}$.
	Let $\mathcal{V}$ denote a non-empty set of pairwise orthogonal elements of $\mathfrak{S}$ and let $W \in\mathfrak{S}$. We say the pair $(W,\mathcal{V})$ is a \emph{container pair} if the following are satisfied:
	\begin{itemize}
		\item for all $V \in \mathcal{V}$, $V\nest W$;
		\item there exists $Q \nest W$ such that $Q \perp V$ for all $V \in \mathcal{V}$.
	\end{itemize}
	Let $\mathfrak{D}$ denote  the set of all container pairs. We will denote a pair $(W,\mathcal{V}) \in \mathfrak{D}$ by $D_W^\mathcal{V}$. 
	
	Let $\mathfrak{R} = \mathfrak{S} \cup \mathfrak{D}$. We will prove $\mathcal{X}$ has a hierarchically hyperbolic space structure with index set $\mathfrak{R}$.  Since $(\mathcal{X},\mathfrak{S})$ is an almost HHS, we can continue to use the spaces, projections, and relations for elements of $\mathfrak{S}$. Thus we only define new projections, relative projections, and relations when elements of $\mathfrak{D}$ are involved. If $D_W^\mathcal{V} \in \mathfrak{D}$, then the associated hyperbolic space, $\fontact{D_W^\mathcal{V}}$, will be a single point.

		\textbf{Projections:} For $D_W^\mathcal{V} \in \mathfrak{D}$, the projection map is just the constant map to the single point in $\fontact{D_W^\mathcal{V}}$.
		
		\textbf{Nesting:} Let  $Q \in \mathfrak{S}$ and  $D_W^\mathcal{V}, D_T^\mathcal{R}\in \mathfrak{D}$. 	
		\begin{itemize}
			\item     Define $Q \nest D_W^\mathcal{V}$ if $Q \nest W$ in $\mathfrak{S}$ and $Q \perp V$ for all $V \in \mathcal{V}$. 
			\item  Define  $D_W^\mathcal{V} \nest Q$ if  $W \nest Q$ in $\mathfrak{S}$. 
			\item  Define $D_W^\mathcal{V} \nest D_T^\mathcal{R}$ if $W \nest T$ in $\mathfrak{S}$  and for all $R \in\mathcal{R}$ either $R \perp W$ or there exists $V \in \mathcal{V}$ with $R \nest V$.
		\end{itemize}
		
		These definitions ensure $\nest$ is still a partial order and maintain the $\nest$--maximal element of $\mathfrak{S}$ as  the $\nest$--maximal element of $\mathfrak{R}$.

		Since the hyperbolic spaces associated to elements of $\mathfrak{D}$ are points, define  $\rho^Q_{D_W^\mathcal{V}} = \fontact{D_W^\mathcal{V}}$ for every $Q \in \mathfrak{R}$ and $D_W^\mathcal{V} \in \mathfrak{D}$ with $Q \propnest D_W^\mathcal{V}$. The downwards relative projection $\rho_Q^{D_W^{\mathcal{V}}} \colon \fontact D_W^{\mathcal{V}} \to \fontact Q$ can be defined arbitrarily.
		
		If $D_W^\mathcal{V} \in \mathfrak{D}$ and $Q \in\mathfrak{S}$ with $D_W^\mathcal{V} \nest Q$, then $V \propnest Q$ in $\mathfrak{S}$ for each $V \in \mathcal{V}$. Thus we define $\rho_Q^{D_W^\mathcal{V}} = \bigcup_{V \in \mathcal{V}} \rho_Q^V$. Lemma \ref{lem:orthongality_and_rel_projections} ensures that $\rho_Q^{D_W^\mathcal{V}}$ has diameter at most $4E$. In this case, we define $\rho_{D_W^{\mathcal{V}}}^{Q} \colon \fontact Q \to \fontact D_W^{\mathcal{V}}$ as the constant map to the single point in $\fontact D_W^{\mathcal{V}}$. 
		
		\textbf{Finite complexity:} 
	First consider a nesting chain of the form $D_{W}^{\mathcal{V}_1} \propnest D_{W}^{\mathcal{V}_2} \propnest \dots \propnest D_{W}^{\mathcal{V}_n}$.  
     
     \begin{claimappend} \label{claim:finite_complexity_common_base_case}
     The length of $D_{W}^{\mathcal{V}_1} \propnest D_{W}^{\mathcal{V}_2} \propnest \dots \propnest D_{W}^{\mathcal{V}_n}$ is bounded above by $E^2+E$.
     \end{claimappend}
     
     \begin{proof}
     For each $V \in \bigcup_{i=1}^{n} \mathcal{V}_i$, we have $V \nest W$ and hence $V \not \perp W$.  As $D^{\mathcal{V}_{i-1}}_{W} \propnest D^{\mathcal{V}_{i}}_{W}$ for each $i\in \{2,\dots,n\}$,  every element of $\mathcal{V}_{i}$ must therefore be nested into an element of $\mathcal{V}_{i-1}$. 
     Denote the elements of $\mathcal{V}_{i}$ by $V^{i}_{1}, \dots V^{i}_{k_{i}}$. Since each $\mathcal{V}_i$ is a pairwise orthogonal subset of $\mathfrak{S}$, we have $k_i \leq E$ for each $i \in \{1,\dots, n\}$ by the bounded pairwise orthogonality axiom of an almost HHS (Definition \ref{almostHHS}). 
    We define a \emph{$\mathcal{V}$--nesting chain} to be a maximal chain of the form $V_{j_{m}}^{m} \nest V_{j_{m-1}}^{m-1} \nest \dots \nest V_{j_{1}}^{1}$ for some $m \in \{1,\dots, n\}$ and $j_{i} \in \{1,\dots,k_i\}$, with $i \in \{1,\dots,m\}$. 
    Since the elements of $\mathcal{V}_i$ are pairwise orthogonal for each $i \in \{1,\dots,n\}$, if $V^m_{j_m}$ is the $\nest$--minimal element of a $\mathcal{V}$--nesting chain, then $V^m_{j_m}$ is nested into exactly one element of $\mathcal{V}_i$ for each $i \leq m$. 
    This implies that each $\mathcal{V}$--nesting chain is determined by its $\nest$--minimal element. Further, the set of $\nest$--minimal  elements of $\mathcal{V}$--nesting chains is pairwise orthogonal. By the bounded pairwise orthogonality axiom of an almost HHS, this implies there exist at most $E$ $\mathcal{V}$--nesting chains.

    In order for $ D_W^{\mathcal{V}_i} \neq D_W^{\mathcal{V}_{i+1}}$, either $k_{i+1} < k_{i}$ or there exists $j_{i} \in \{1, \dots, k_i\}$, $j_{i+1} \in \{1, \dots, k_{i+1}\}$ such that $V_{j_{i+1}}^{i+1} \propnest V_{j_{i}}^{i}$.  Thus, every step up the chain $D_{W}^{\mathcal{V}_1} \propnest D_{W}^{\mathcal{V}_2} \propnest \dots \propnest D_{W}^{\mathcal{V}_n}$ results in either a strict decrease in $k_i$ (the cardinality of $\mathcal{V}_{i}$)  to $k_{i+1}$ (the cardinality of $\mathcal{V}_{i+1}$) or a strict step down one of the $\mathcal{V}$--nesting  chains. Note that $k_{i}$ may increase when we encounter a strict step down one of the $\mathcal{V}$--nesting chains, since multiple elements of $\mathcal{V}_{i+1}$ may be nested into the same element of $\mathcal{V}_{i}$. Such an increase in $k_{i}$ corresponds to the nesting chain branching into multiple chains, which may only happen at most $E-k_{1}$ times, as there are at most $E$ $\mathcal{V}$--nesting chains.
     Hence, the length of $D_{W}^{\mathcal{V}_1} \propnest D_{W}^{\mathcal{V}_2} \propnest \dots \propnest D_{W}^{\mathcal{V}_n}$ is bounded by $k_{1}+(E-k_{1})=E$ plus the total number of times a strict decrease can occur across all of the $\mathcal{V}$--nesting chains.
  
    Each $\mathcal{V}$--nesting chain $V_{j_{m}}^{m} \nest V_{j_{m-1}}^{m-1} \nest \dots \nest V_{j_{1}}^{1}$ contains at most $E$ distinct elements of $\mathfrak{S}$ by the finite complexity of $\mathfrak{S}$. Bounded pairwise orthogonality implies there are at most $E$ different $\mathcal{V}$--nesting chains, thus the number of steps of the chain $D_{W}^{\mathcal{V}_1} \propnest D_{W}^{\mathcal{V}_2} \propnest \dots \propnest D_{W}^{\mathcal{V}_n}$ where there is a strict decrease within one of the $\mathcal{V}$--nesting chains is at most $E^{2}$. This bounds the length of $D_{W}^{\mathcal{V}_1} \propnest D_{W}^{\mathcal{V}_2} \propnest \dots \propnest D_{W}^{\mathcal{V}_n}$ by $E^2 + E$.
    \end{proof}
     
    We now consider a nesting chain of the form $D_{W_1}^{\mathcal{V}_1} \propnest D_{W_2}^{\mathcal{V}_2} \propnest \dots \propnest D_{W_n}^{\mathcal{V}_n}$. In this case, $W_1 \nest W_2 \nest \dots \nest W_n$, but not all of these nestings must be proper.
     Let $1 = i_1 < i_2 < \dots < i_k$ be the minimal subset of $\{1,\dots,n\}$ such that if $ i_j \leq i < i_{j+1}$, then $W_{i_j} = W_i$. Thus $W_{i_1} \propnest W_{i_2} \propnest \dots \propnest W_{i_k}$, and $k \leq E$ by finite complexity of $\mathfrak{S}$.  Claim \ref{claim:finite_complexity_common_base_case} established that $|i_j - i_{j+1}| \leq E^2 + E
     $, so $n \leq k(E^2+E) \leq E^{3}+E^{2}$, that is, any $\propnest$--chain of elements of $\mathfrak{D}$ has length at most $E^3+E^{2}$.

     Finally, since any $\propnest$--chain of elements of $\mathfrak{R}$ can be partitioned into a $\propnest$--chain of elements of $\mathfrak{D}$ and a $\propnest$--chain of elements of $\mathfrak{S}$, any  $\propnest$--chain in $\mathfrak{R}$ has length at most $E^3 +E^{2}+E$.

		\textbf{Orthogonality:} Two elements $D_W^\mathcal{V}, D_T^\mathcal{R} \in \mathfrak{D}$ are orthogonal if $W \perp T$ in $\mathfrak{S}$.  Let  $Q \in \mathfrak{S}$ and  $D_W^\mathcal{V} \in \mathfrak{D}$. Define $Q \perp D_{W}^\mathcal{V}$ if, in $\mathfrak{S}$, either $W \perp Q$  or $Q \nest V$ for some $V \in \mathcal{V}$. These definitions, plus the definition of nesting imply for all $W,V,Q \in \mathfrak{R}$, if $W \perp V$ and $Q \nest V$, then $W \perp V$. We now verify that $\mathfrak{R}$  satisfies the container requirements of the orthogonality axiom.
		
		Let $W,V \in \mathfrak{S}$ with $V \propnest W$ and $\{Q \in \mathfrak{R}_W: Q \perp V\} \neq \emptyset$, i.e.,  $(W,\{V\})$ is a container pair. In this case, the container of $V$ in $W$ for $\mathfrak{R}$ is $D_W^{\{V\}}$. 
		We now show containers exist for situations involving elements of $\mathfrak{D}$. We split this into three subcases.

		\textbf{Case 1: $\mathbf{D_W^\mathcal{V} \in \mathfrak{D}}$ and $\mathbf{Q \in \mathfrak{S}}$ with $\mathbf{D_W^\mathcal{V} \nest Q}$.}  Since $(W,\mathcal{V})$ is a container pair, there exists $P \in \mathfrak{S}$ with $P \nest W$ and $V \perp P$ for all $V \in\mathcal{V}$. Suppose that $D^{\mathcal{V}}_{W}$ requires a container in $Q$, that is, there is an element $U$ of $\mathfrak{R}$ that is orthogonal to $D^{\mathcal{V}}_{W}$ and nested in $Q$. We verify that $(Q,\{P\})$ is a container pair and $D_Q^{\{P\}}$ is a container of $D_W^\mathcal{V}$ in $Q$. 
		
		If $T \in \mathfrak{S}$ with $T \perp D_W^\mathcal{V}$ and $T \nest Q$, then  $T \perp W$ or $T \nest V$ for some $V \in \mathcal{V}$. 
		In either case, we have $T \perp P$, so $(Q,\{P\})$ is a container pair and $T \nest D_Q^{\{P\}}$. If $D_T^\mathcal{R} \in \mathfrak{D}$ with $D_T^\mathcal{R} \nest Q$ and $D_T^\mathcal{R} \perp D_W^\mathcal{V}$, then $T \perp W$ and $T \nest Q$. Since $P \nest W$, this implies $T \perp P$ and so $(Q,\{P\})$ is again a container pair, and $D_T^\mathcal{R} \nest D_Q^{\{P\}}$.

		\textbf{Case 2: $\mathbf{D_W^\mathcal{V}, D_T^\mathcal{R} \in \mathfrak{D}}$ where $\mathbf{D_W^\mathcal{V} \nest D_T^\mathcal{R}}$.} Since $(W,\mathcal{V})$ is a container pair, there exists $P \in \mathfrak{S}$ so that $P \nest W$ and $P \perp V$ for all $V \in \mathcal{V}$. Since $D_W^\mathcal{V} \nest D_T^\mathcal{R}$, it follows that for all $R \in \mathcal{R}$, either $R \perp W$  or  there exists $V \in \mathcal{V}$ so that $R\nest V$. 
		In both cases, $R \perp P$. Thus  $\mathcal{P} = \mathcal{R} \cup \{P\} $ is a pairwise orthogonal collection of elements of $\mathfrak{S}$. Suppose that $D^{\mathcal{V}}_{W}$ requires a container in $D^{\mathcal{R}}_{T}$, that is, there is an element $U$ of $\mathfrak{R}$ that is orthogonal to $D^{\mathcal{V}}_{W}$ and nested in $D^{\mathcal{R}}_{T}$. We verify that $(T,\mathcal{P})$ is a container pair and $D_T^{\mathcal{P}} \propnest D_T^\mathcal{R}$  is a container for $D_W^\mathcal{V}$ in $D_T^\mathcal{R}$.

		If $Q \in \mathfrak{S}$ satisfies $Q \nest D_T^\mathcal{R}$ and $D_W^\mathcal{V} \perp Q$, then  either $Q \perp W$ or $Q \nest V$ for some $V \in \mathcal{V}$.
		In both cases, $Q \perp P$. Further, we must have  $Q \nest T$ and $Q \perp R$ for each $R \in \mathcal{R}$ as $Q \nest D_T^{\mathcal{R}}$. Thus $(T,\mathcal{P})$ is a container pair and $Q \nest D_T^{\mathcal{P}}$.  On the other hand, if $D_Q^\mathcal{Z} \in \mathfrak{D}$ satisfies $D_Q^\mathcal{Z} \perp D_W^\mathcal{V}$ and $D_Q^\mathcal{Z} \nest D_T^\mathcal{R}$, then $Q \perp W$, $Q \nest T$, and for each $R \in \mathcal{R}$ either $R \perp Q$ or there exists $Z \in \mathcal{Z}$ with $R \nest Z$. Since $(Q,\mathcal{Z})$ is a container pair, there exists $U \in \mathfrak{S}$ such that $U \nest Q$ and $U \bot Z$ for all $Z \in \mathcal{Z}$. Since $Q \perp W$, we also have $U \perp P$ as $U \nest Q$ and $P \nest W$.  For each $R \in \mathcal{R}$, either $R \perp Q$ or there exists $Z \in \mathcal{Z}$ with $R \nest Z$. In both cases, $R \perp U$. Thus, $U$ is orthogonal to all elements of $\mathcal{P}= \mathcal{R} \cup \{P\}$ and moreover $U \nest Q \nest T$, so $(T,\mathcal{P})$ is a container pair. Furthermore, $D_Q^\mathcal{Z} \nest D_T^{\mathcal{P}} = D_T^{\mathcal{R} \cup \{P\}}$ since $D_Q^\mathcal{Z} \nest D_T^{\mathcal{R}}$ and $P \bot Q$. We have therefore shown that $D_T^{\mathcal{P}}$  is a container for $D_W^\mathcal{V}$ in $D_T^\mathcal{R}$.
		
		\textbf{Case 3: $\mathbf{D_T^\mathcal{R} \in \mathfrak{D}}$  and $\mathbf{Q \in \mathfrak{S}}$ with $\mathbf{Q \nest D_T^\mathcal{R}}$.} This implies $\mathcal{Q} = \mathcal{R} \cup \{Q\}$ is a pairwise orthogonal set of elements of $\mathfrak{S}$. Further, suppose that $Q$ requires a container in $D^{\mathcal{R}}_{T}$, that is, there is an element of $\mathfrak{R}$ that is orthogonal to $Q$ and nested in $D^{\mathcal{R}}_{T}$.
		We verify that $(T,\mathcal{Q})$ is a container pair and  $D_T^\mathcal{Q}$ is a  container for $Q$ in $D_T^\mathcal{R}$. 
		
		Suppose there exists $V \in \mathfrak{S}$ with  $V \nest D_T^{\mathcal{R}}$ and $V\perp Q$. Then $V \nest T$ and $V$ is orthogonal to all the elements of $\mathcal{R}\cup \{Q\}$. 
		Thus $(T,\mathcal{Q})$ is a container pair, so $D_T^\mathcal{Q}$ exists and $V \nest D_T^{\mathcal{Q}}$. Now suppose there exists  $D_W^\mathcal{V} \nest D_T^{\mathcal{R}}$ such that $D_W^{\mathcal{V}} \perp Q$.
		Since $(W,\mathcal{V})$ is a container pair, there exists $U \in \mathfrak{S}$ with $U \nest W$ and $U$ orthogonal to each element of $\mathcal{V}$. As $D_W^\mathcal{V} \nest D_T^{\mathcal{R}}$, for each $R \in \mathcal{R}$ either $R \perp W$ or there exists $V \in \mathcal{V}$ such that $R \nest V$. In both cases, $R \perp U$. Further, as $Q \perp D_W^{\mathcal{V}}$, we have $Q \perp W$ or $Q \nest V$ for some $V \in \mathcal{V}$. In both cases, $Q \perp U$. Therefore $U$ is orthogonal to every element of $\mathcal{Q}$, and moreover $U \nest W \nest T$ since $D^{\mathcal{V}}_{W} \nest D^{\mathcal{R}}_{T}$. Thus $(T,\mathcal{Q})$ is a container pair and $U \nest D^{\mathcal{Q}}_{T}$. 
		Now, for each $R \in \mathcal{R}$, either $R \perp W$ or $R \nest V$ for some $V \in \mathcal{V}$. Since $\mathcal{Q} = \mathcal{R} \cup \{Q\}$ and $Q \bot W$, this implies $D_W^\mathcal{V} \nest D_T^\mathcal{Q}$. 
		Thus, $(T,\mathcal{Q})$ is a container pair and  $D_T^\mathcal{Q}$ is a  container for $Q$ in $D_T^\mathcal{R}$.

		\textbf{Transversality:} An element of $\mathfrak{R}$ is transverse to an element of $\mathfrak{D}$ whenever it is not nested or orthogonal.  Since the hyperbolic spaces associated to elements of $\mathfrak{D}$ are points, we only need to define the relative projections  from an element of $\mathfrak{D}$ to an element of $\mathfrak{S}$.
		Let $D_W^\mathcal{V} \in \mathfrak{D}$ and $Q \in \mathfrak{S}$ and suppose $D_W^\mathcal{V} \trans Q$. This implies $W \not \perp Q$ and $ W \not \nest Q$. We define $\rho_Q^{D_W^\mathcal{V}}$ based on the $\mathfrak{S}$--relation between $Q$ and the elements of $\mathcal{V}$.
		
		\begin{itemize}
			\item If $Q \perp V$ for all $V \in \mathcal{V}$, then $Q \not \nest W$ as $Q \nest W$ would imply $Q \nest D_W^\mathcal{V}$. Thus we must have $Q \trans W$, so we define $\rho_Q^{D_W^{\mathcal{V}}} = \rho_Q^W$. 
			\item  If $V \trans Q$ or $V \propnest Q$ for some $V \in \mathcal{V}$, then $\rho_Q^V$ exists and we define $\rho_Q^{D^\mathcal{V}_W}$ to be the union of all the $\rho_Q^V$ for $V \in\mathcal{V}$ with $V \trans Q$ or $V \propnest Q$. Lemma \ref{lem:orthongality_and_rel_projections} ensures $\rho_Q^{D_W^\mathcal{V}}$ has diameter at most $4E$ in this case.
			\item  If $Q \nest V$ for some $V$, then $Q \perp D_W^\mathcal{V}$ which contradicts $Q \trans D_W^\mathcal{V}$, so this case does not occur. 
		\end{itemize}

	\textbf{Consistency:}
	Since the only elements of $\mathfrak{R}$ whose associated spaces are not points are in $\mathfrak{S}$, the first two inequalities in the consistency axiom for  $(\mathcal{X},\mathfrak{S})$ imply the same two inequalities for $(\mathcal{X},\mathfrak{R})$. To verify the final clause of the consistency axiom, we need to check that if $Q ,R, T \in \mathfrak{R}$ such that $Q \propnest R$  with $\rho_T^R$ and $\rho_T^Q$ both defined, then  $\dist_T( \rho_T^Q, \rho_T^R)$ is uniformly bounded in terms of $E$.
	We can assume $T \in \mathfrak{S}$ as $\fontact{T}$ has diameter zero otherwise. We can further assume at least one of $Q$ and $R$ is an element of $\mathfrak{D}$, as we already have the consistency axiom for elements of $\mathfrak{S}$.
	
	\textbf{Case 1:} $\mathbf{Q \propnest}$ $\mathbf{R \propnest T.}$
	\begin{itemize}
		\item Assume $Q \in \mathfrak{S}$ and $R = D_W^\mathcal{V} \in \mathfrak{D}$. Fix $V \in \mathcal{V}$. Since $D_W^\mathcal{V}=R \nest T$ and $\rho_T^{D_W^\mathcal{V}} = \bigcup_{U \in \mathcal{V}}\rho_T^U$, we have $\rho_T^V \subseteq \rho_T^{D_W^\mathcal{V}} = \rho_T^R$. Since $V \perp Q$, Lemma \ref{lem:orthongality_and_rel_projections}  says $\dist_T(\rho_T^R,\rho_T^Q)\leq \dist_T(\rho_T^V,\rho_T^Q) \leq 2E$.

		\item Assume $Q = D_W^\mathcal{V} \in \mathfrak{D}$ and $R \in \mathfrak{S}$. Fix $V \in \mathcal{V}$. In this case, $\rho_T^V \subseteq \rho_T^Q$ since $D_W^\mathcal{V}=Q \propnest T$. Since $D_W^\mathcal{V} =Q \propnest R$, we have $V \propnest W \propnest R$. Thus, the consistency axiom for $\mathfrak{S}$ says $\dist_T(\rho_T^Q,\rho_T^R) \leq \dist_T(\rho_T^V,\rho_T^R) \leq E$.
		
		\item Assume $Q = D_W^\mathcal{V} \in \mathfrak{D}$ and $R = D_{W'}^{\mathcal{V}'}  \in \mathfrak{D}$. Thus $W \nest W' \propnest T$ and consistency in $\mathfrak{S}$ implies $ \dist_T(\rho_T^W,\rho_T^{W'}) \leq E$. Fix $V \in \mathcal{V}$ and $V' \in \mathcal{V}'$. Consistency in $\mathfrak{S}$ also implies $\dist_T(\rho_T^V,\rho_T^W) \leq E$ and  $\dist_T(\rho_T^{V'},\rho_T^{W'}) \leq E$. Since $\rho_T^V \subseteq \rho_T^Q$ and $\rho_T^{V'} \subseteq \rho_T^R$, we have $\dist_T(\rho_T^Q,\rho_T^R) \leq \dist_T(\rho_T^V,\rho_T^{V'}) \leq \dist_T(\rho_T^V,\rho_T^W)+\diam(\rho_T^W)+\dist_T(\rho_T^W,\rho_T^{W'})+\diam(\rho_T^{W'})+\dist_T(\rho_T^{W'},\rho_T^{V'}) \leq 5E$. 
	\end{itemize}
	
	\textbf{Case 2:} $\mathbf{Q \propnest R}$\textbf{,} $\mathbf{R \trans T}$\textbf{, and} $\mathbf{Q \not\perp T.}$ In this case we have either $Q \trans T$ or $Q \propnest T$.
	\begin{itemize}
		\item Assume $Q \in \mathfrak{S}$ and $R = D_W^\mathcal{V} \in \mathfrak{D}$. Since $D_W^\mathcal{V} = R$ is transverse to $T$ we cannot have $T \nest V$ for any $V \in \mathcal{V}$ (this would imply $D_W^\mathcal{V} \perp T$). If $V \perp T$ for all $V \in \mathcal{V}$, then $W \trans T$ (as shown in the proof of transversality) and $\rho_T^R = \rho_T^{D_W^{\mathcal{V}}} = \rho_T^W$. Since $Q \nest R = D_W^\mathcal{V}$, we have $Q \nest W$ and consistency in $\mathfrak{S}$ implies $\dist_T(\rho_T^Q,\rho_T^R) = \dist_T(\rho_T^Q,\rho_T^W) \leq E$. If instead there exists $V\in\mathcal{V}$ so that $V \trans T$ or $V \propnest T$, then $\rho_T^V \subseteq \rho_T^{D_W^\mathcal{V}} = \rho_T^R$. Since $Q \nest R = D_W^\mathcal{V}$,  $Q \perp V$ and Lemma \ref{lem:orthongality_and_rel_projections} gives $\dist_T(\rho_T^Q,\rho_T^R) \leq \dist_T(\rho_T^Q,\rho_T^V) \leq 2E$.

		\item Assume $Q = D_W^\mathcal{V} \in \mathfrak{D}$ and $R \in \mathfrak{S}$. As before, $T \not\nest V$ for all $V \in \mathcal{V}$. First assume there exists $V \in \mathcal{V}$ so that $V \trans T$ or $V \propnest T$. This occurs when either $D_W^\mathcal{V} =Q \propnest T$ or $Q \trans T$ and not every element of $\mathcal{V}$ is orthogonal to $T$. In both cases, $\rho_T^V \subseteq \rho_T^{D_W^\mathcal{V}} = \rho_T^Q$ and consistency in $\mathfrak{S}$ implies $\dist_T(\rho_T^Q,\rho_T^R) \leq \dist_T(\rho_T^V,\rho_T^R) \leq 2E$ because   $V \nest W \propnest R$. Now assume $T \perp V$ for all $V \in \mathcal{V}$. This can only occur when $D_W^\mathcal{V}=Q$ is transverse to $T$. In this case, $W \trans T$ and  $\rho_T^Q = \rho_T^{D_W^\mathcal{V}} = \rho_T^W$. Since $W \nest R$, consistency in $\mathfrak{S}$ implies  $\dist_T(\rho_T^R,\rho_T^Q) = \dist_T(\rho_T^R,\rho_T^W) \leq E$. 
		
		\item Assume $Q = D_W^\mathcal{V}\in \mathfrak{D}$ and $R = D_{W'}^{\mathcal{V}'}  \in \mathfrak{D}$. As before, $T \not\nest V$ for all $V \in \mathcal{V} \cup \mathcal{V}'$. If $\rho_T^R = \rho_T^{W'}$, then we have the first case of transversality, that is, $W' \trans T$ and $V' \perp T$ for all $V' \in \mathcal{V}$. 
		Thus, if $\rho_T^R =\rho_T^{W'}$, then the result reduces to the previous  bullet, replacing $R$ with $W'$. We can therefore assume  $\rho_T^R \neq \rho_T^{W'}$, meaning we have the second case of transversality where there exists $V' \in \mathcal{V}'$ so that $V'$ is either transverse to or properly nested into $T$. 
		
		Suppose $\rho^{Q}_{T} \neq \rho^{W}_{T}$ too. This implies there also exists $V \in \mathcal{V}$  so that  $V$ is either transverse to or properly nested into $T$. 
		Furthermore, $\rho_T^V \subseteq \rho_T^Q$ and $\rho_T^{V'} \subseteq \rho_T^{R}$. Now,  $D_W^{\mathcal{V}} \nest D_{W'}^{\mathcal{V}'}$ implies $V' \perp W$ or $V'$ is nested into an element of $\mathcal{V}$. 
		If $V' \perp W$, then $V \perp V'$ and Lemma \ref{lem:orthongality_and_rel_projections} implies $\dist_T(\rho_T^Q,\rho_T^R) \leq \dist_T(\rho_T^V,\rho_T^{V'})\leq 2E$. If $V'$ is nested into an element of $\mathcal{V}$, then either $V' \nest V$ or $V' \perp V$ since $\mathcal{V}$ is a pairwise orthogonal subset of $\mathfrak{S}$. 
		By applying consistency in $\mathfrak{S}$ when $V' \nest V$ or Lemma \ref{lem:orthongality_and_rel_projections} when $V' \perp V$, we have  $\dist_T(\rho_T^Q,\rho_T^R) \leq \dist_T(\rho_T^V,\rho_T^{V'} )\leq 2E$.
		
		Now suppose $\rho^{Q}_{T} = \rho^{W}_{T}$. Then $D_W^{\mathcal{V}} \nest D_{W'}^{\mathcal{V}'}$ implies $V' \perp W$ or $V'$ is nested into $W$. Applying Lemma \ref{lem:orthongality_and_rel_projections} if $V' \bot W$, or consistency in $\mathfrak{S}$ if $V' \nest W$, we again obtain $\dist_T(\rho_T^Q,\rho_T^R) \leq \dist_{T}(\rho^{W}_{T},\rho^{V'}_{T}) \leq 2E$.
	\end{itemize}

	\textbf{Uniqueness, bounded geodesic image, large links:} Since the only elements of $\mathfrak{R}$ whose associated spaces are not points are in $\mathfrak{S}$, these axioms for  $(\mathcal{X},\mathfrak{R})$ follow from the fact that they hold in $(\mathcal{X},\mathfrak{S})$.

	\textbf{Partial realization:}   Let $T_1, \dots , T_n$ be pairwise orthogonal elements of $\mathfrak{R}$, and let $p_i \in \fontact{T_i}$ for each $i \in \{1,\dots,n\}$. Without loss of generality,  assume $T_1,\dots,T_k \in \mathfrak{S}$ and $T_{k+1},\dots,T_n \in \mathfrak{D}$ where $k \in \{0,\dots,n\}$. If $k=0$ (resp. $k=n$), then each $T_i \in \mathfrak{D}$ (resp. $T_i \in \mathfrak{S}$).

	For  $i \in \{k+1,\dots, n\}$, let $T_i = D_{W_i}^{\mathcal{V}_i}$ and let $q_i$ be any point in $\rho_{W_i}^{D_{W_i}^{\mathcal{V}_i}} \subseteq  \fontact{W_i}$.  Since $T_1,\dots,T_n$ are pairwise orthogonal, it follows that $W_{k+1},\dots,W_n$ are pairwise orthogonal too, and for each $j \in \{1,\dots,k\}$, $T_{j}$ is either nested into an element of some $\mathcal{V}_{i_{j}}$ or orthogonal to all $W_{k+1},\dots,W_{n}$. Without loss of generality, assume that $T_{1},\dots,T_l$ are nested into elements of $\mathcal{V}_{m+1}\cup\dots\cup\mathcal{V}_{n}$ and $T_{l+1},\dots,T_k,W_{k+1},\dots,W_{n}$ are pairwise orthogonal, where $l \leq k$, $m \leq n$, and $n-m \leq l$. If $l = 0$, then $n=m$ and each $T_j$ is orthogonal to every $W_i$.
	Otherwise, for each $j \in \{1,\dots,l\}$, $T_j$ is nested in some $W_{i}$ for $i \in \{m+1,\dots,n\}$. In both cases, $T_{1},\dots,T_{k},W_{k+1},\dots,W_{m}$ are pairwise orthogonal elements of $\mathfrak{S}$. We can therefore use the partial realization axiom in $\mathfrak{S}$ on the points  $p_1,\dots,p_k,q_{k+1},\dots,q_m$ to produce a point $x \in \mathcal{X}$ with the following properties:
	
	\begin{enumerate}
		\item $\dist_{T_i}(x,{p_i}) \leq E$ for $i \in \{1,\dots,k\}$; \label{item:partial_realization_1}
		\item $\dist_{W_i}(x,{q_i}) \leq E$ for $i \in \{k+1,\dots,{m}\}$; \label{item:partial_realization_2}
		\item  for all $i \in \{1,\dots,k\}$ if $Q \trans {T_i}$ or ${T_i} \propnest Q$, then $\dist_Q(x,\rho_Q^{T_i}) \leq E$; \label{item:partial_realization_3}
		\item for all $i \in \{k+1,\dots,{m}\}$ if $Q \trans {W_i}$ or ${W_i} \propnest Q$, then $\dist_Q(x,\rho_Q^{W_i}) \leq E$. \label{item:partial_realization_4}
	\end{enumerate}
	
	Now, for $Q \in \mathfrak{S}$, define $b_Q \in \fontact{Q}$ as follows.  Let $\mathcal{V} = \bigcup_{{i=k+1}}^{n} \mathcal{V}_i$ and $\mathcal{V}_Q = \{V \in \mathcal{V} : V \trans Q \text{ or } V\propnest Q\}$. If  $\mathcal{V}_{Q} \neq \emptyset$, then define $b_Q$ to be any point in $\bigcup_{V \in \mathcal{V}_{Q}} \rho_Q^V$. Since $\mathcal{V}$ is a collection of pairwise orthogonal elements of $\mathfrak{S}$, the diameter of $\bigcup_{V \in \mathcal{V}_Q} \rho_Q^V$ is at most $2E$ by Lemma \ref{lem:orthongality_and_rel_projections}. If either $Q \nest V$ for some $V \in \mathcal{V}$ or $Q \perp V$ for all $V \in\mathcal{V}$ then define $b_Q = \pi_Q(x)$. Since $\mathcal{V}$ is a collection of pairwise orthogonal elements of $\mathfrak{S}$, these two cases encompass all elements of $\mathfrak{S}$.
	
	\begin{claimappend}
		The tuple  $(b_Q)_{Q\in\mathfrak{S}}$ is  $3E$--consistent.
	\end{claimappend}
	
	\begin{proof}
		Let $R,Z \in \mathfrak{S}$. Recall that if $b_Z =\pi_Z(x)$ and $b_R = \pi_R(x)$, then the $E$--consistency inequalities for $b_R$ and $b_Z$ are satisfied by the consitency axiom of $(\mathcal{X}, \mathfrak{S})$. Thus we can assume that there exists  $V \in \mathcal{V}$ so that either $V \propnest Z$ or $V \trans Z$. Fix $V \in \mathcal{V}$ so that $b_Z \in \rho_Z^V$. We need to verify the consistency inequalities when $R \trans Z$, $R \propnest Z$, and $Z \propnest R$.
		
		\textbf{Consistency when} $\mathbf{R\trans Z}$\textbf{:} Assume $R \trans Z$. If $R \perp V$, $V \nest R$, or $R \nest V$ then either Lemma \ref{lem:orthongality_and_rel_projections} or consistency in $\mathfrak{S}$ implies $\dist_Z(\rho_Z^{V},\rho_Z^R) \leq 2E$. Since $b_Z \in \rho_Z^{V}$, we have $\dist_Z(b_Z,\rho_Z^R) \leq 3E$. Now suppose $R \trans V$ so that $\mathcal{V}_{R}$ is non-empty. In this case, $b_R \in \bigcup_{U \in \mathcal{V}_{R}} \rho_R^U$ and so $b_{R}$ is within $2E$ of $\rho_R^V$. Now, if $\dist_Z(b_Z,\rho_Z^R) > 3E$, then $\dist_Z(\rho_Z^{V},\rho_Z^R) > 2E$. Thus $\rho$--consistency (Lemma \ref{lem:rho_consistency}) implies $\dist_R(\rho_R^V,\rho_R^Z) \leq E$. It follows that $\dist_R(b_R, \rho_R^Z) \leq 3E$ by the triangle inequality.
		
		\textbf{Consistency when} $\mathbf{R \propnest Z}$\textbf{:} Assume $R \propnest Z$. As before, if $R \perp V$, $V \nest R$, or $R \nest V$ then $\dist_Z(\rho_Z^{V},\rho_Z^R) \leq 2E$ and we have $\dist_Z(b_Z,\rho_Z^R) \leq 3E$. Thus, we can assume $R \trans V$ so that $b_R$ is within $2E$ of $\rho_R^V$. Now, if $\dist_Z(b_Z, \rho_Z^R) >3E$, then $\dist_Z(\rho_Z^V,\rho_Z^R) > 2E$, and $\rho$--consistency implies $\diam(\rho_R^V \cup \rho_R^Z(\rho_Z^V)) \leq E$. However, this implies $\diam(b_R \cup \rho_R^Z(b_Z)) \leq 3E$ since $b_Z \in \rho_Z^V$ and $\dist_R(b_R,\rho_R^V) \leq 2E$.

		\textbf{Consistency when} $\mathbf{Z\propnest R}$\textbf{:} Assume $Z \propnest R$. If $R$ is orthogonal to all elements of $\mathcal{V}$, then $R \perp V$ implies $V\perp Z$ which contradicts the assumption that $V \propnest Z$ or $V \trans Z$. On the other hand, if there  exists $V' \in \mathcal{V}$ so that $R \nest V'$, then either $R \perp V$ (if $V' \perp V$) or $R \nest V$ (if $V'=V$). But this implies either $V \perp Z$ or $Z \propnest V$, both of which give a contradiction if $V \trans Z$ or $V \propnest Z$. There must therefore be an element of $\mathcal{V}$ that is either properly nested in or transverse to $R$, and we can repeat the same argument as in the previous case, switching the roles of $R$ and $Z$.
	\end{proof}
	
	Let $y\in \mathcal{X}$ be the point produced by applying the realization theorem (Theorem \ref{thm:realization}) in $\mathfrak{S}$ to the tuple $(b_Q)$. We claim $y$ is a partial realization point for $p_1,\dots,p_n$ in $\mathfrak{R}$.  Since $\fontact{D_{W_i}^{\mathcal{V}_i}}$ is a single point, $y$ satisfies the first requirement of the partial realization axiom in $\mathfrak{R}$ for $p_{k+1},\dots,p_n$.  For $i \leq k$, $T_i$ is either nested into an element of $\mathcal{V}_{m+1}\cup\dots\cup\mathcal{V}_{n}$ or orthogonal to all
	$W_{k+1},\dots,W_{n}$. 
	This implies $T_i$ is either nested into an element of $\mathcal{V}$ or orthogonal to all elements of $\mathcal{V}$. In both cases $b_{T_{i}} = \pi_{T_{i}}(x)$, and we have that $\pi_{T_i}(y)$ is uniformly close to $\pi_{T_i}(x)$, which is in turn $E$--close to $p_i$.
	
	Now, let $Q \in \mathfrak{S}$ with $Q \trans T_i$ or $T_i \propnest Q$ for some $i \in\{1,\dots,n\}$. We verify $\dist_Q(y,\rho_Q^{T_i})$ is uniformly bounded when $i \leq k$ and $i > k$ separately.

	Assume $i \leq k$, so that $T_i \in \mathfrak{S}$. If {$i \leq k$} and $b_Q = \pi_Q(x)$, then $\dist_Q(y,\rho_Q^{T_i})$ is bounded by item (\ref{item:partial_realization_3}). If {$i \leq k$} and $b_Q \neq \pi_Q(x)$, then $b_{Q} \in \rho^{V}_{Q}$ for some $V \in \mathcal{V}$ and $T_i$ is either orthogonal to or nested into $V$. If $T_{i} \bot V$ then $\dist_Q(b_Q,\rho_Q^{T_i}) \leq 3E$ by Lemma \ref{lem:orthongality_and_rel_projections}. If $T_{i} \propnest V$ then $\dist_{Q}(b_{Q},\rho^{T_{i}}_{Q}) \leq 2E$ by consistency. The result then follows from the triangle inequality since $\pi_Q(y)$ is uniformly close to $b_Q$.
	
	Now assume $i > k$, so that $T_i = D_{W_i}^{\mathcal{V}_i} \in \mathfrak{D}$. If $D_{W_i}^{\mathcal{V}_i} \propnest Q$, then $\rho_Q^{V} \subseteq \rho_Q^{D_{W_i}^{\mathcal{V}_i}}$ for all $V \in\mathcal{V}_i$. 
	Since $b_Q$ is within $2E$ of any $\rho_Q^V$ for $V \in \mathcal{V}_i$, this bounds $\dist_Q(y,\rho_Q^{D_{W_i}^{\mathcal{V}_i}})$ uniformly. On the other hand, if $ D_{W_i}^{\mathcal{V}_i} \trans Q$, then either $Q \perp V$ for all $V \in \mathcal{V}_i$ or  there exists $V \in\mathcal{V}_i$ so that $V \trans Q$ or $V \propnest Q$.  In the latter case, $\rho_Q^V \subseteq \rho_Q^{D_{W_i}^{\mathcal{V}_i}}$ and we are finished since $b_Q$ is within $2E$ of $\rho_Q^V$, giving a uniform bound on the distance from $\pi_Q(y)$ to $\rho_Q^{D_{W_i}^{\mathcal{V}_i}}$. In  the former case, we must have $W_i \trans Q$ and  $\rho_Q^{D_{W_i}^{\mathcal{V}_i}}$ is equal to $\rho_Q^{W_i}$. If $b_Q = \pi_Q(x)$ than we are done by item (\ref{item:partial_realization_4}).
	Otherwise, there exists  $V' \in \mathcal{V} - \mathcal{V}_i$ so that $V' \trans Q$ or $V' \propnest Q$ and $b_Q \in \rho_Q^{V'}$. Since $V' \perp W_i$, it follows that $\rho_Q^{V'}$ is within $2E$ of $\rho_Q^{W_i}$. Thus $b_Q$, and hence $\pi_Q(y)$, is uniformly close to $\rho_Q^{W_i}=\rho_Q^{D_{W_i}^{\mathcal{V}_i}}$.
	This concludes the proof of Theorem \ref{thm:almost_HHS_are_HHS}.
\end{proof}

\begin{remappend}\label{rem:almost_HHGs_are_HHSs}
	We say $G$ is an almost HHG if there exists an almost HHS $(\mathcal{X}, \mathfrak{S})$ such that $G$ and $(\mathcal{X},\mathfrak{S})$ satisfy the definition of a hierarchically hyperbolic group where `HHS' is replaced with `almost HHS'.
	 The above proof shows that if $(G,\mathfrak{S})$ is an almost HHG, then the structure $\mathfrak{R}$ from Theorem \ref{thm:almost_HHS_are_HHS} is an HHG structure for $G$.
\end{remappend}

 The following corollary gives criteria for the HHS structure from Theorem \ref{thm:almost_HHS_are_HHS} to have unbounded products. This is the version of Theorem \ref{thm:almost_HHS_are_HHS} that is applied in Theorem \ref{thm:betterHHSstructure} to prove that every hierarchically hyperbolic space with the bounded domain dichotomy admits an HHS structure with unbounded products.

\begin{corappend} \label{cor:unbddproducts}
  Let $(\mathcal{X},\mathfrak{T})$ be an almost HHS with the bounded domain dichotomy. If for every non-$\nest$--maximal domain $V \in \mathfrak{T}$, there exist $W,Q \in \mathfrak{T}$ so that $W\nest V$, $Q \perp V$, and $\diam(\fontact{W}) = \diam(\fontact{Q}) = \infty$, then the HHS structure $\mathfrak{R}$ obtained by applying Theorem \ref{thm:almost_HHS_are_HHS} to $\mathfrak{T}$ has unbounded products. 
\end{corappend}

\begin{proof}
Assume for every non-$\nest$--maximal domain $V \in \mathfrak{T}$, 
there exist $W,Q \in \mathfrak{T}$ so that $W\nest V$, $Q \perp V$ 
and $\diam(\fontact{W}) = \diam(\fontact{Q}) = \infty$. Let $\mathfrak{R}$ be the HHS structure obtained from $\mathfrak{T}$ using Theorem \ref{thm:almost_HHS_are_HHS}.  If $V \in \mathfrak{T}$ and $V$ is not $\nest$--maximal, then the above property implies that $\FU{V}$ and $\EU{V}$ are both infinite diameter. Thus, we need only verify unbounded products for elements of $\mathfrak{R} - \mathfrak{T}$. Using the notation of Theorem \ref{thm:almost_HHS_are_HHS}, let $D=D_W^\mathcal{V} \in \mathfrak{R} - \mathfrak{T}$ and assume $\diam(\FU{D}) = \infty$.  Now, $V \perp D_W^{\mathcal{V}}$ for all $V \in \mathcal{V}$, and by construction of $\mathfrak{T}$, there exists $Q \in \mathfrak{T}$ so that $Q \nest V$ and $\diam(\fontact{Q}) = \infty$. Since $Q \perp D_W^{\mathcal{V}}$, this implies  $\diam(\EU{D}) = \infty$. Therefore $(\mathcal{X},\mathfrak{R})$ is an HHS with unbounded products. 
\end{proof}